\documentclass{amsart}
\usepackage{marktext}
\usepackage{gimac}
\newdelim{\ip}{\langle}{\rangle}

\newcommand{\RR}{\mathbb{R}}
\newcommand{\ZZ}{\mathbb{Z}}

\newcommand{\TT}{\mathbb{T}}

\newcommand{\PP}{\mathbb{P}}

\newcommand{\pa}{\partial}
\newcommand{\cT}{\mathcal{T}}
\newcommand{\cL}{\mathcal{L}}

\newcommand{\bu}{\mathbf{u}}

\newcommand{\bv}{\mathbf{v}}
\newcommand{\bk}{\mathbf{k}}
\newcommand{\tbk}{\widetilde{\mathbf{k}}}

\newcommand{\rL}{\mathring{L}}

\newcommand{\rH}{\mathring{H}}
\newcommand{\tX}{\widetilde{X}}
\newcommand{\ep}{\epsilon}

\newcommand{\what}{\widehat}
\newcommand{\ds}{\displaystyle}

\usepackage{verbatim}
\usepackage{esint}
\usepackage{bm}
\usepackage{enumitem}
\usepackage{amsmath,amsthm}

\usepackage{color}

\begin{document}
\title[Global existence for the advective KSE]{Global existence for the two-dimensional Kuramoto-Sivashinsky equation with advection}

\author{Yuanyuan Feng}
\email{yzf58@psu.edu}
\author{Anna L. Mazzucato}
\email{alm24@psu.edu}
\address{Department of Mathematics, Penn State University, University Park, PA
16802, USA}

\begin{abstract} 
We study the Kuramoto-Sivashinsky equation (KSE) in scalar form on the two-dimensional torus with and without advection by an incompressible vector field. We prove local existence  of mild solutions for arbitrary data in $L^2$. We then study the issue of global existence. We prove global existence for the KSE in the presence of advection for arbitrary data, provided the advecting velocity field $\bv$ satisfies certain conditions that ensure the dissipation time of the associated hyperdiffusion-advection equation is sufficiently small. In the absence of advection, global existence can be shown only if the linearized operator does not admit any growing mode and for sufficiently small initial data.
\end{abstract}



\maketitle

\section{Introduction}

We consider the Kuramoto-Sivashinsky equation (KSE for short) with and without advection on a two-dimensional torus
$\TT^2=[0, L_1]\times [0, L_2]$ with periodic boundary conditions. The KSE arises in combustion and is a model for diffusive instabilities of flame fronts and, more generally, for phenomena with large-scale instabilities. (We refer to \cite{HN86} for a more in-depth discussion and historical perspective on its derivation.)

Two forms of the equation go under the name KSE, a scalar form for a potential function $\phi$ and a vector, differentiated form for its gradient $\bu=\nabla \phi$, defined in any space dimension $d$. In the scalar form, the KSE is the following hyperviscous semilinear equation:
\begin{align}\label{eq:kse}
 \partial_t \phi +\frac{1}{2}|\nabla\phi|^2=-\Delta^2 \phi -\Delta \phi.
\end{align}
In the differentiated form, the KSE  becomes:
\begin{align}\label{eq:kse2}
 \partial_t \bu +\frac{1}{2}  \nabla |\bu|^2=-\Delta^2 \bu -\Delta \bu,
\end{align}
and, owing to the fact $\bu$ is curl free, the non-linearity can be written in advection form as $\bu\cdot \nabla \bu$, as in Burger's equation.
We will confine to the scalar form, since we will consider the addition of transport by a given, incompressible flow:
\begin{align}\label{e:aks}
\partial_t \phi +\bv\cdot \nabla\phi +\frac{1}{2}|\nabla\phi|^2=-\Delta^2 \phi -\Delta \phi,
\end{align}
where $\bv$ is a given, time-dependent, divergence-free vector field. The KSE with advection has been used to model passive flame propagation in premixed-combustion for example \cite{DK95}.

In this work, we are concerned with the long-time existence of solutions to \eqref{eq:kse} and (implicitly) of \eqref{e:aks}. A main difficulty in proving global existence is the lack of a maximum principle for the KSE, due to the presence of the biharmonic operator. In dimension $d=1$, it is possible to obtain an {\em a priori} estimate on the $L^2$ norm of $u$ due to the special structure of the non-linearity in \eqref{eq:kse2} , which can be written in divergence form. The $L^2$ control  allows to prove global existence by a standard continuation argument \cite{Tad86}. The stability of the zero solution has also been established \cite{Good94}, \cite{NST85}. Differently than for the incompressible Navier-Stokes or Euler equations, this is no longer the case in dimension $d>1$. In fact, there is no known global estimate for any $L^p$ norm. It should be noted that, even in dimension $d=1$, the $L^2$ norm can grown exponentially fast, because the linearized operator
$\pa_t +\Delta^2+\Delta$ admits exponentially growing modes for large enough periods $L_1, \; L_2$. For $d>1$, the issue of global existence of solutions to the KSE is still open. We only consider the classical form of KSE, and not the modified or generalized KSE, for which more results are known (see e.g. \cite{GMP08, Mol00}).

There is an extensive literature concerning the KSE in dimension $d=1$, concerning also analyticity of solutions (see  \cite{CEES93,Gru00} and references therein), and optimal bounds on the growth of the $L^2$ norm as a function of the period $L$ \cite{BG06,GF19,GJO15,GO05,Otto09} (see also \cite{SSArXiv07}). There are a few results for the two and multi-dimensional KSE.  Short time existence and analyticity is known to hold in the full space with data in certain $L^p$ spaces for \eqref{eq:kse} \cite{BS07} (see also \cite{IS16}). There are fewer results dealing with global existence. For $d=2$, global existence holds for thin domains \cite{SellT92,BKRZ14} , and for the anisotropically reduced KSE \cite{LY20}. In \cite{AM19}, the second author and Ambrose studied the differentiated form \eqref{eq:kse2} on  $\TT^2$ and proved short-time existence and analyticity with a bound on the analyticity radius for data in $L^2$ and in the Wiener algebra. They established global existence
of mild solutions for mean-free data $\bu(0)$ sufficiently small in $L^2$ and in the Wiener algebra, but only in the absence of growing modes, which happens when $L_1,\;
 L_2<2/\pi$.

Here, we extend the results in \cite{AM19} in two ways. We prove existence of mild solutions for sufficiently small data $\phi(0)$ in $L^2$ on an arbitrary interval of time $[0,T]$ (that is, $\bu_0$ small in the homogeneous Sobolev space $\dot{H}^{-1}$) for the scalar form \eqref{eq:kse} of KSE in the absence of growing modes, which requires more refined semigroup estimates than for \eqref{eq:kse2}, given the more singular nature of the non-linearity. The mean of the solution is not preserved by the forward evolution as for \eqref{eq:kse2}, but it can be controlled. We also consider the KSE with linear advection and prove that global existence can be achieved in the presence of growing modes and for arbitrary data, if the advecting field $\bv$ is relaxation enhancing with sufficiently small dissipation time. An example of such a flow is an exponentially mixing flow with sufficiently large amplitude. Informally, given a dissipative system, we define its {\em dissipation time} $\tau^\ast$ as the time it takes the system to dissipate a fixed amount of its initial energy (for a precise definition, see Definition \ref{d:dissipationTime}.  Starting with the seminal work of Constatin {\em et Al.} \cite{CKRZ08} (see also \cite{C-ZDEArXiv18,FI19}), it has been recognized that fast advection can have a regularizing effect. Such an effect has been used to prevent blow-up in a number of physical models, such as aggregation models \cite{IXZArXiv19, He18, HT19, KX16} and  reactive flows \cite{CNR08}. There is also an important connection with inviscid damping for incompressible fluids, which we do not discuss in detail (see \cite{BC-T17,Gal18,GNRS20} and references therein).
We should  mention that there are other known mechanisms for stabilization in dissipative equations, such as fast rotation and dispersive effects  (see the recent work  \cite{KTZ18} and references therein).
In \cite{FFITArXiv20}, the first author, Feng, Iyer, and Thiffeault studied the effect of mixing on phase separation in binary mixtures modeled by the Cahn-Hilliard equation, which has the same linear part as the KSE. There they established enhanced dissipation for the advection-hyperdiffusion operator $-\Delta^2  -\bv\cdot \nabla$, when $\bv$ is a strongly mixing incompressible flow.

We confine ourselves to the two-dimensional KSE to avoid being overly technical and to exemplify the effects of advection and mixing, but we expect that similar results hold for the three-dimensional KSE as well. 

Throughout the paper, we use standard notation to represent function spaces. In particular, $H^s(\TT^2)$, $s\in \RR$, is the standard $L^2$-based Sobolev space, and $C([0,\infty))$  is the space of continuous and bounded functions on $[0,\infty)$ with the sup norm. $C$ denotes a generic constant that may change from line to line and depend on $L_1, L_2$. 
If $X$ is a function space on $\TT^2$, $\mathring{X}$ will denote the subspace of all mean-free functions belonging to $X$.

We close the Introduction with an outline of the paper. In Section \ref{s:local}, we study the KSE with advection in the presence of growing modes. We first prove short-time existence of mild solution for arbitrary data in $L^2$, and then show that such solution can be continued for all times, provided the advecting velocity field $\bv$ satisfies certain conditions that ensure the dissipation time of the associated hyperdiffusion-advection equation is sufficiently small. In Section  \ref{s:global}, we contrast this result with global existence for KSE without advection, which we can establish only in the absence of growing modes for sufficiently small initial data.

\section*{Acknowledgments} The authors thank Gautam Iyer for useful discussions. A.M. was partially supported by the US National Science Foundation grants DMS-1909103 and DMS-1615457. Part of this work was conducted while the second author was on leave from Penn State University to New York University-Abu Dhabi.

\section{The KSE with advection} \label{s:local}

In this section we study the KSE in scalar form  with advection \eqref{e:aks} on the torus $\TT^2$. We do not impose any restriction on the  periods, and consequently the linearized operator exhibits exponentially growing modes.  For local existence, we need to assume only that the advecting field  $\bv\in L^\infty((0,\infty);L^2(\TT^2)$, and that $\mathrm{div}\,\bv=0$ in distributional sense. To study global existence, we will assume that $\bv$ is  Lipschitz continuous in space uniformly in time.

Given a function $f\in L^p(\TT^2)$, $p\geq 1$, we denote by $\what{f}(\bk)$  the Fourier coefficient of $f$ at frequency $\bk \in \ZZ^2$. For notational ease as in \cite{AM19}, we let \ $\tbk :=  2\pi
\big(\frac{k_1}{L_1},\frac{k_2}{L_2}\big)$, where $\bk=(k_1,k_2)\in \ZZ^2$, so
$\tbk\in \widetilde\ZZ^2:=2\pi L_1^{-1}\ZZ\times 2\pi L_2^{-1}\ZZ$. We also set \, $\kappa := |\tbk|$.

With slight abuse of notation we denote the k-th Fourier coefficient of $f\in L^2(\TT^2)$ by \ $\what{f}(\tbk)$. Then, we can define equivalent norms in the Sobolev space $H^s(\TT^2)$ and in the homogeneous Sobolev space $\Dot{H}^s(\TT^2)$, $s\in \RR$, by
\begin{align}
      \|f\|^2_{H^s} &=  \sum_{\tbk \in \widetilde \ZZ^2} (1+|\tbk|^2)^{s}\, |\what{f}(\tbk)|^2=
      \|(I-\Delta)^{s/2}\,f\|^2_{L^2},
       \label{eq:HsNorm}\\
       \|f\|^2_{\Dot{H}^s} &=  \sum_{\tbk \in \widetilde \ZZ^2, \tbk\ne \mathbf{0}}  |\tbk|^{2s}\, |\what{f}(\tbk)|^2= \|(-\Delta)^{s/2}\,f\|^2_{L^2}, \label{eq:HsHomNorm}
\end{align}
where $(-\Delta)^{s/2}$ agrees with the Fourier Multiplier with symbol $\kappa^s$, $\kappa \ne 0$ and $I$ is the identity operator. We observe that $f\in H^s \Leftrightarrow f\in \dot H^s$, if $f\in L^2$ and $s\geq 0$.

Given $f\in L^p(\TT^2)$, $p\geq 1$, we define the projection $\PP$ onto the space of mean-zero functions over the torus, that is:
\[
       \what{\PP(f)}(\mathbf{0}) = 0,
\]
and let $\Bar{\phi}$ denote the average of $\phi$ over the torus. This projection is orthogonal on $L^2$,   bounded on the Sobolev space $H^s(\TT^2)$ for $s>0$, and  commutes with any Fourier multiplier.
We then set
\begin{equation} \label{eq:OspacesDef} 
  \rL^2(\TT^2)=\PP(L^2(\TT^2)), \qquad \rH^s(\TT^2)=\PP(H^s(\TT^2)), \quad s> 0.
\end{equation}
The norm in  $\rH^s(\TT^2)$ is equivalent to the seminorm  in the homogeneous
Sobolev spaces $\dot{H}^s(\TT^2)$.

We begin by recalling the notion of mild and weak solutions, which we adapt to our setting, and introduce some short-hand notation that will be used throughout the paper. We define $\cL=\Delta^2+\Delta$, an operator on $L^2$ with domain $H^4(\TT^2)$, and denote with $e^{-t \cL}$ the semigroup generated by $\cL$ on $L^2$, which is explicitly given by:
\[
    e^{-t\cL} f = \mathcal{F}^{-1} ( e^{-t (|\tbk|^4-|\tbk|^2)}\, \what{f})= \mathcal{F}^{-1} ( e^{-t \, \sigma(\tbk) }\, \what{f}),
\]
where $\mathcal{F}$ denotes the Fourier Transform on $\TT^2$ and 
\[
    \sigma(\tbk):= |\tbk|^4-|\tbk|^2=\kappa^4-\kappa^2.
\]
It was shown in \cite{AM19} that $e^{-t\cL}$ generates a $C^0$ and analytic semigroup on $L^2(\TT^2)$, and the following operator norm bounds hold:
\begin{equation} \label{eq:SemigroupBound}
 \begin{aligned}
   \|(-\Delta)^{s/2} e^{-t \cL} f\|_{L^2(\TT^2)} &\leq C\, e^{\frac{t}{2}}
\max\big(1,t^{-\frac{s}{4}}\big) \,
   \|e^{-\frac{t}{2} \cL} f\|_{L^2(\TT^2)} \\
   &\leq C\, e^{t} \max\big(1,t^{-\frac{s+1}{4}\big)}\,
   \|f\|_{L^1(\TT^2)} , \qquad   s>0, \; t>0,
 \end{aligned}
\end{equation}
where $C=C(L_1,L_2,s)$. We briefly recall the simple proof of this estimate  below in Lemma \ref{l:l2l1} and \ref{l:gradestimate} for completeness, and as a starting point for more refined estimates done in Section \ref{s:global}.

We also formally set:
\begin{equation} \label{eq:Bdef}
  B(\phi_1, \phi_2) :=-\frac{1}{2}\int_0^te^{-(t-\tau)\mathcal{L}} \nabla \phi_1(\tau) \cdot \nabla \phi_2(\tau)\,d\tau,
\end{equation}
and
\begin{equation} \label{eq:Ldef}
  L(\phi) := - \int_0^t e^{-(t-\tau)\mathcal{L}}(v(\tau)\cdot \nabla \phi(\tau))\,d\tau.
\end{equation}
Lastly, $\phi(t)$ means the function of $x$, $\phi(t)(x)=\phi(t,x)$.

\begin{definition} \label{d:solutionDef}
A  function $\phi:C([0,T];L^2(\TT^2))$, $0<T\leq \infty$, such that $\nabla \phi$ is locally integrable, is called  a {\em mild solution} of \eqref{e:aks}  on $[0,T]$ with initial data $\phi_0\in L^2(\TT^2)$, if it satisfies
\begin{align} \label{eq:MildDef}
\phi(t)=\mathcal{N}(\phi) (t)&:=e^{-t\mathcal{L}}\phi_0-\frac{1}{2}\int_0^t e^{-(t-\tau)\mathcal{L}}\abs{\nabla\phi (\tau)}^2\,d\tau-\int_0^t e^{-(t-\tau)\mathcal{L}}(\bv\cdot \nabla \phi)\,d\tau \nonumber\\
&=e^{-t\mathcal{L}}\phi_0+B(\phi, \phi)(t) + L(\phi)(t), \qquad 0\leq t\leq T,
\end{align}
pointwise in time with values in $L^2$, where the integral is intended in the B\"ochner sense (see e.g. \cite{Pazy}).\newline
A function $\phi \in L^\infty([0,T];L^2(\TT^2))\cap L^2([0,T];H^2(\TT^2))$ is called a {\em weak solution} of  \eqref{e:aks}  on $[0,T)$ with initial value $\phi(0)=\phi_0\in L^2(\TT^2)$ if, for all $\varphi \in C^\infty_c([0,T)\times\TT^2)$,
\begin{align} \label{eq:weakDef}
\int_{\TT^2} \phi_0\varphi(0)\,dx+& \int_{0}^T\int_{\TT^2} \phi\, \pa_t\varphi \, dx\, dt=  \int_{0}^T\int_{\TT^2}  \Delta \phi\, \Delta \varphi \, dx\, dt\nonumber  - 
  \int_{0}^T\int_{\TT^2}  \nabla \phi\, \nabla \varphi \, dx\, dt  \\
  & + \frac{1}{2} \int_{0}^T\int_{\TT^2}   |\nabla \phi|^2 \,\varphi\,dx\,dt+\int_{0}^T\int_{\TT^2}  \varphi\,\bv\cdot \nabla \phi \, dx\, dt\,,
\end{align}
and $\partial_t\phi\in L^2([0 , T];H^{-2}(\T^2))$.
\end{definition}

Mild solutions are formally fixed points of the non-linear map $\mathcal{N}$, and \eqref{eq:MildDef} is in the form of a Volterra integral equation. A standard way to obtain mild solutions is therefore to apply the Banach Contraction Mapping Theorem.

\subsection{Short-time existence with data in $L^2$} \label{s:localExistence}
 
 We establish the existence of solutions to \eqref{eq:MildDef} on a small time interval $[0,T]$, $0<T\leq 1$, for arbitrary initial data $\phi_0\in L^2(\TT^2)$, by proving that the map $\mathcal{N}$ is a contraction in a suitable adapted Banach space $\tilde{X}_T$. Because of the presence of growing modes, one does not obtain global existence for small data by this method.   
 We need to introduce an adapted space, among other reasons, to make sense of the nonlinearity, which requires $\nabla \phi$ to be locally integrable (with or without advection). By contrast, in the vector form \eqref{eq:kse2} of the KSE, the non-linearity is a well-defined distribution if $u(t)$ is in $L^2(\TT^2)$ pointwise in time, and one can prove local existence of solutions just using the space $C([0,T];L^2)$ for the contraction argument (cf. \cite{AM19}).

Given $0<T<1 $, we define the space
\begin{align}
X_{T} : =\set{\phi: \T^2\times \R_{+}\to \R \st  \sup_{0< t\leq T} t^{\frac{1}{4}}\, \norm{\nabla \phi}_{L^2} < \infty}\,,
\end{align}
and let
\begin{equation}
   \tX_T := C([0, T]; L^2(\T^2)) \cap X_T, 
\end{equation}
which is a Banach space equipped with the norm:
\begin{align}
\norm{\phi}_{\tX_T}=\max(\sup_{0\leq t\leq T}\norm{\phi}_{L^2}\,,\sup_{0<t\leq T}t^{\frac{1}{4}}\norm{\nabla \phi}_{L^2}).
\end{align}


We will verify that $\mathcal{N}$ is a contraction map on a ball in $\tX_T$, using the bounds on the semigroup $e^{-t\cL}$. 
For completeness, we briefly recall the proof of the two needed estimates, which we will refine in Section \ref{s:global} to obtain global existence for \eqref{eq:kse} in the absence of growing modes.
We define the set  $S=\set{\tbk\in \tilde \Z^2 \st \abs{\tbk}^2\geq \frac{1}{2}\abs{\tbk}^4}$, then $S$ is a finite set. When $\tbk\in S$ , we have $|\tbk|\leq c_0$ and $\sigma(\tbk)\geq -1/4$. Otherwise, it holds that $\sigma(\tbk) \geq |\tbk|^4/2$.

\begin{lemma}\label{l:l2l1}
For any $0<t<1$, there exists a constant $C=C(L_1,L_2)$ such that 
\begin{align}
\norm{e^{-t\mathcal{L}}f}_{L^2}\leq C t^{-\frac{1}{4}}\norm{f}_{L^1}\,.
\end{align}
\end{lemma}

\begin{proof}
By definition and Plancherel's identity, we have
\[
  \begin{aligned}
    \norm{e^{-t\mathcal{L}} &f}_{L^2}^2=\sum_{\tbk\in \tilde \Z^2} e^{-2t\,\sigma(\tbk)} \, |\hat f(\tbk)|^2 \leq (\sup_{\tbk\in    
    \tilde \Z^2} |\what f(\tbk)|)^2 \sum_{\tbk\in \tilde \Z^2} e^{-2t\, \sigma(\tbk)}\\
    &\leq  C\, \norm{f}_{L^1}^2(\sum_{\tbk\in S} e^{-2t\,\sigma(\tbk)}  +\sum_{\tbk \in \tilde \Z^2\setminus S} e^{-2t\,\sigma(\tbk)})
     \leq C\, \norm{f}_{L^1}^2 (e^{t/2}+\int_{\R^2} e^{-t|x|^4}\,dx) \\
     &\leq  C\, \norm{f}_{L^1}^2(e^{t/2}+t^{-1/2}) \leq C\, t^{-1/2}\norm{f}_{L^1}^2\,.
  \end{aligned}
\]
\end{proof}

\begin{lemma}\label{l:gradestimate}
For any $s>0$ and $0<t<1$, there exists a constant $C=C(L_1,L_2)$ such that
\begin{align}
\norm{(-\Delta)^{s/2}e^{-t\mathcal{L}} f}_{L^2} \leq C \,t^{-s/4} \, \norm{f}_{L^2}\,.
\end{align}
\end{lemma}

\begin{proof}
Again, by Plancherel's identity: 
\[
   \begin{aligned}
 \norm{(-\Delta)^{s/2}e^{-t\mathcal{L}}f}_{L^2}^2
&=\sum_{\tbk\in S} |\tbk|^{2s} e^{-2t \sigma(\tbk)} |\what f(\tbk)|^2+\sum_{\tbk\in \tilde \Z^2\setminus S} \abs{\tbk}^{2s}e^{-2t \sigma(\tbk)} |\what f(\tbk)|^2\\
&\leq C\, \left(e^{t/2}\sum_{\tbk\in S} |\what f(\tbk)|^2+ \sum_{\tbk\in \tilde \Z^2\setminus S} |\tbk|^{2s} \,e^{-t|\tbk|^4}\, 
|\what f (\tbk)|^2\right)\\
&\leq C\, \left[ e^{t/2}\sum_{\tbk\in S} |\hat f(\tbk)|^2+ \paren[\Big]{\sup_{x\in \R_{+}}e^{-tx^4}x^{2s}}\sum_{\tbk\in \tilde \Z^2-S}
 |\what f(\tbk)|^2\right] \\
&\leq  C\, \left(e^{t/2} + t^{-s/2}\right) \norm{f}_{L^2}^2 \leq C\, t^{-s/2}\, \norm{f}_{L^2}^2\,.
  \end{aligned}
\]
\end{proof}

We now state the main result of this section.

\begin{theorem}  \label{t:localAKS}
 Let $\phi_0\in L^2(\TT^2)$ and let $\bv\in L^\infty(\R_+;L^2(\TT^2))$. There exists $0<T\leq 1$ depending on $L_1, L_2$, $\sup_{t>0}\|\bv\|_{L^2}$, and on $\|\phi_0\|_{L^2}$ such that \eqref{e:aks} admits a mild solution 
 $\phi$ on $[0,T]$, which is unique in $\tX_T$.
\end{theorem}

We split the proof in several parts.  

\begin{lemma}\label{l:close}
 Let $0<T\leq 1$. The map $\mathcal{N}: \tX_T\to \tX_T$ and there exists $C=C(L_1,L_2)>0$,  such that 
\begin{equation}\label{e:Fgrowth}
   \norm{\mathcal  N (\phi)}_{\tX_T} \leq C\,\left(\norm{\phi_0}_{L^2}+T^{1/4}\, \norm{\phi}_{\tX_T}^2+T^{1/2}\,\norm{\bv}_{L^\infty(
   \R _+;L^2)}\,\norm{\phi}_{\tX_T}\right)\,.
\end{equation}
\end{lemma}  

\begin{proof}
The proof follows by establishing the following two claims:
\begin{itemize}[]
\item[\textit{Claim 1.}] If $\phi\in \tX_T$, then $\mathcal{N}(\phi) \in C([0,T];L^2(\T^2))$.
\item[\textit{Claim 2.}] If $\phi\in \tX_T$, then $\sup_{0<t\leq T}t^{1/4}\norm{\nabla( \mathcal{N}(\phi))}_{L^2} <\infty$.
\end{itemize}
First, we observe that the fact that $e^{-t\cL}$ generates a strongly continuous semigroup and the norm estimates in Lemma \ref{l:l2l1} and \ref{l:gradestimate} imply that $e^{-t\cL} \phi_0$ is in $\tX_T$ for any fixed $T>0$. Next, from the definition and properties of the B\"ochner integral (see e.g. \cite{Pazy}), the integral  on the right-hand side of \eqref{eq:MildDef} is well defined and belongs to $C([0,T];L^2(\TT^2))$ provided the $L^2(\TT^2)$ norm of the terms under the integral sign belongs to $L^1((0,T))$.
For any $0<t\leq T$, we have, in fact, again from the semigroup estimates:
\[
\begin{aligned}
\norm{\mathcal{N}(\phi)(t)}_{L^2}&\leq C\,\left(\norm{\phi_0}_{L^2}+\int_0^t\norm{e^{-(t-\tau)\mathcal{L}}\, (|\nabla \phi(\tau)|^2+ \bv(\tau)\cdot \nabla \phi(\tau))}_{L^2}\,d\tau\right)\\
&\leq C\, \left(\norm{\phi_0}_{L^2}+\int_0^t(t-\tau)^{-1/4}\norm{|\nabla \phi(\tau)|^2 +\bv(\tau)\cdot \nabla \phi(\tau)}_{L^1}\,d\tau \right)\\
&\leq C\,\left(\norm{\phi_0}_{L^2}+\int_0^t (t-\tau)^{-1/4} \tau^{-1/2} (\tau^{1/4}\norm{\nabla \phi(\tau)}_{L^2})^2\,d\tau\right.\\
&\left. \quad \qquad+\int_0^t(t-\tau)^{-1/4}\,\tau^{-1/4} \norm{\bv(\tau)}_{L^2}(\tau^{1/4}\,\norm{\nabla \phi(\tau)}_{L^2})\,d\tau\right)\\
&\leq C\,\left(\norm{\phi_0}_{L^2}+ t^{1/4}\int_0^1(1-\bar\tau)^{-1/4} \bar\tau^{-1/2} \norm{ \phi}_{\tX_T}^2\,d\bar\tau\right.\\
&\left. \quad \qquad+t^{1/2}\int_0^1(1-\bar\tau)^{-1/4}\,\bar\tau^{-1/4} \norm{\bv}_{L^\infty(\R_+;L^2)}\,\norm{\phi}_{\tX_T})\,d\bar\tau\right) \\
&\leq C\, \left( \norm{\phi_0}_{L^2}+t^{1/4}\norm{\phi}_{\tX_T}^2+t^{1/2}\norm{\bv}_{L^\infty(\R_+;L^2)}\,
\norm{\phi}_{\tX_T}\right)\,,
\end{aligned}
\]
where $\Bar\tau=\tau/t$. This proves \textit{Claim 1} taking the supremum over $t$.

 Similarly, from the properties of the semigroup generated by $\cL$, we have:
\[
\begin{aligned}
\norm{\nabla (\mathcal {N}(\phi))(t)}_{L^2} &\leq \norm{e^{-t\mathcal{L}}\nabla \phi_0}_{L^2}+\frac{1}{2}\int_0^t\norm{\nabla e^{-\frac{t-\tau}{2}\mathcal{L}}}_{L^2\to L^2}\, \norm{e^{-\frac{t-\tau}{2}\mathcal{L}}|\nabla \phi(\tau)|^2}_{L^2}\,d\tau\\
&\quad \qquad+\int_0^t\norm{\nabla e^{-\frac{t-\tau}{2}\mathcal{L}}}_{L^2\to L^2}\, \norm{e^{-\frac{t-\tau}{2}\mathcal{L}}
\,(\bv(\tau)\cdot \nabla \phi(\tau))}_{L^2}\,d\tau\\
&\leq C \, t^{-1/4}\,\norm{\phi_0}_{L^2}+C \,\int_0^t (t-\tau)^{-1/2}\,\tau^{-1/2}(\tau^{1/4}\norm{\nabla \phi(\tau)}_{L^2})^2\,d\tau\\
&\quad\qquad+C \int_0^t (t-\tau)^{-1/2}\,\tau^{-1/4}\norm{\bv(\tau)}_{L^2}(\tau^{1/4}\norm{\nabla\phi(\tau)}_{L^2})\,d\tau\\
&\leq C\, t^{-1/4}\,\norm{\phi_0}_{L^2}+C \left(\int_0^1(1-\Bar\tau)^{-1/2} \,\Bar\tau^{-1/2} \,d\Bar\tau\right) \norm{\phi}_{\tX_T}^2\\
&\quad\qquad+C \,\left( t^{1/4}\int_0^1(1-\Bar\tau)^{-1/2}\,\Bar\tau^{-1/4}\,d\Bar\tau\right)\norm{\bv}_{L^\infty(\R_+;L^2)}\norm{\phi}_{\tX_T}\\
&\leq C \left(t^{-1/4}\,\norm{\phi_0}_{L^2}+\norm{\phi}_{\tX_T}^2+t^{1/4}\norm{\bv}_{L^\infty(\R_+;L^2)}\norm{\phi}_{\tX_T}\right),
\end{aligned}
\]
where $\Bar\tau=\tau/t$. Consequently 
\begin{equation}
\sup_{0<t\leq T}t^{1/4}\norm{\nabla (\mathcal{N}(\phi))}_{L^2} \leq C\norm{\phi_0}_{L^2}+CT^{1/4}\norm{\phi}_{\tX_T}^2+CT^{1/2}\norm{\bv}_{L^\infty(\R_+;L^2)}\norm{\phi}_{\tX_T}\,.
\end{equation}
This proves \textit{Claim 2} and thus concludes the proof of this lemma.
\end{proof}

 Next, we prove that the map $\mathcal{N}$ is Lipschitz on $\tX_T$ with a constant that depends on $T$.

\begin{lemma}\label{l:diff}
Let $0<T\leq 1$. There exists a constant $C=C(L_1,L_2)$ such that, for any $\phi_1,\phi_2 \in \tX_T$,
\begin{equation}\label{e:diff}
  \begin{aligned}
  \norm{\mathcal N (\phi_1)-& \mathcal{N}(\phi_2)}_{\tX_T}\leq \\ \nopagebreak[4]
   &C\,T^{1/4}\,
 \big(\norm{\phi_1}_{\tX_T}+\norm{\phi_2}_{\tX_T}+T^{1/4}\,\norm{\bv}_{L^\infty(\R_+;L^2)}\big)
 \,\norm{ \phi_1-\phi_2}_{\tX_T}.
 \end{aligned}
\end{equation}

\end{lemma}

\begin{proof}
The proof is similar to that of Lemma \ref{l:close}, and follows again from the following two claims.
Let  $\phi_i \in \tX_T$, $i=1,2$. For all  $0\leq t\leq T$:
\begin{itemize}[]
\item[\textit{Claim 1.}]
\begin{multline}
\norm{\mathcal{N}(\phi_1)(t)-\mathcal{N}(\phi_2)(t)}_{L^2}\leq \\
 C\,t^{1/4}\,(\norm{\phi_1}_{\tX_T}+\norm{\phi_2}_{\tX_T}+t^{1/4}\norm{\bv}_{L^\infty(\R_+;L^2)})\, \norm{\phi_1-\phi_2}_{\tX_T}\,;
\end{multline}
\item[\textit{Claim 2.}]
\begin{multline}
 t^{1/4}\, \norm{\nabla\mathcal{N}(\phi_1)(t)-\nabla \mathcal{N}(\phi_2)(t)}_{L^2}\leq  \\
 C\,t^{1/4}\,(\norm{\phi_1}_{\tX_T}+\norm{\phi_2}_{\tX_T}+t^{1/4}\norm{\bv}_{L^\infty(\R_+;L^2)})\, \norm{\phi_1-\phi_2}_{\tX_T}.
\end{multline}
\end{itemize}
Since $B(\cdot, \cdot)$ is bilinear and using the properties of the semigroup $e^{-t\cL}$, we have:
\[
\begin{aligned}
\norm{\mathcal{N}&(\phi_1)(t)-\mathcal{N}(\phi_2)(t)}_{L^2} \leq 
 C\, \int_0^t (t-\tau)^{-1/4}\, \left(\norm{(\nabla \phi_1(\tau-\nabla\phi_2(\tau))\cdot\nabla \phi_1(\tau)}_{L^1} \right. \\
 &\left.+
\norm{(\nabla \phi_1(\tau)-\nabla \phi_2(\tau))\cdot\nabla \phi_2(\tau)}_{L^1} 
+\norm{\bv(\tau)\cdot(\nabla \phi_1(\tau)-\nabla \phi_2(\tau))}_{L^1}\right)\,d\tau\\
&\leq C \, \left(\int_0^t (t-\tau)^{-1/4}\, \tau^{-1/2} \, d\tau \right)\, \left[\norm{\phi_1}_{\tX_T} +\norm{\phi_2}_{\tX_T} \right]\, \norm{\phi_1-\phi_2}_{\tX_T} \\
&\quad \qquad+C\,\left(\int_0^t (t-\tau)^{-1/4}\, \tau^{-1/4}\,d\tau\right)\, \norm{\bv}_{L^\infty(\R_+;L^2)} \,
\norm{\phi_1- \phi_2}_{\tX_T}\\
&\leq C\, t^{1/4} \,\big(\norm{\phi_1}_{\tX_T}+\norm{\phi_2}_{\tX_T}+t^{1/4}\norm{\bv}_{L^\infty(\R_+;L^2)}\big)\, \norm{\phi_1-\phi_2}_{\tX_T}\,
\end{aligned}
\]
where we used the change of variable $\Bar\tau=\tau/t$. \textit{Claim 1}  now follows immediately.

Similarly:
\[
\begin{aligned}
\norm{\nabla\mathcal{N}&(\phi_1)(t)-\nabla \mathcal{N}(\phi_2)(t)}_{L^2} \\&\leq 
 C \int_0^t (t-\tau)^{-1/2}\, \left(\norm{(\nabla \phi_1(\tau-\nabla\phi_2(\tau))\cdot\nabla \phi_1(\tau)}_{L^1} \right. \\
 &\left.+
\norm{(\nabla \phi_1(\tau)-\nabla \phi_2(\tau))\cdot\nabla \phi_2(\tau)}_{L^1} 
+\norm{\bv(\tau)\cdot(\nabla \phi_1(\tau)-\nabla \phi_2(\tau))}_{L^1} \right)\,d\tau\\
&\leq C \, \left(\int_0^t (t-\tau)^{-1/2}\, \tau^{-1/2} \, d\tau \right)\, \norm{\phi_1-\phi_2}_{\tX_T} \left[\norm{\phi_1}_{\tX_T} +\norm{\phi_2}_{\tX_T} \right]\\
&\qquad+C\,\left(\int_0^t (t-\tau)^{-1/2}\, \tau^{-1/4}\,d\tau\right)\, \norm{\bv}_{L^\infty(\R_+;L^2)} \,
\norm{\phi_1- \phi_2}_{\tX_T}\\
&\leq C \,\big(\norm{\phi_1}_{\tX_T}+\norm{\phi_2}_{\tX_T}+t^{1/4}\norm{\bv}_{L^\infty(\R_+;L^2)}\big)\, \norm{\phi_1-\phi_2}_{\tX_T},
\end{aligned}
\] 
where the last inequality follows by making again  the change of variable $\Bar\tau=\tau/t$. {\em Claim~2} now follows,
We conclude the proof by taking the supremum  over $0<t\leq T$ in both claims.
\end{proof}

Now we are ready to use  the Banach Fixed Point Theorem to prove the local existence of a mild solution.


\begin{proof}[Proof of Theorem \ref{t:localAKS}]
We denote by $B(0,M)$ the closed ball in $\tX_T$ with center the origin. We let  $M=2C\norm{\phi_0}_{L^2}$, where $C$ is the maximum of the constants appearing in \eqref{e:Fgrowth} and~\eqref{e:diff}, and assume that 
\begin{align} \label{eq:Tchoice}
T\leq \min\Big(1, ~\frac{1}{16(CM+C\norm{v}_{L^\infty(\R_+;L^2)})^4}\Big).
\end{align}
Then, Lemma \ref{l:close} implies that
\begin{align}
\norm{\mathcal{N}(\phi)}_{\tX_T}\leq M, \quad \forall \phi \in B(0, M)\,.
\end{align}
We can further check that with, such such choice of $T$, Lemma \ref{l:diff} gives
\begin{align}
\norm{\mathcal{N}(\phi) -\mathcal{N}(\psi)}_{\tX_T}\leq q \norm{\phi-\psi}_{\tX_T}\,,
\end{align}
where 
\begin{align}
q=\frac{2M+\norm{\bv}_{L^\infty(\R_+;L^2)}}{2M+2\norm{\bv}_{L^\infty(\R_+;L^2)}} <1\,.
\end{align}
By the Banach Contraction Mapping Theorem, there is a unique fixed point of $\mathcal{N}$ in $B(0, M)$. By Definition \ref{d:solutionDef}, $\phi$ is a mild solution of \eqref{e:aks} with initial data $\phi(0)=\phi_0$, which is unique in $\tX_T$.
Indeed,  if there is another mild solution $\tilde \phi$ in $B(0, \tilde M)$ with $\tilde M>M$, then $\tilde \phi=\phi$ in  $B(0,M)\subset B(0,\tilde M)$.
\end{proof}

\begin{corollary} \label{c:contPrinciple}
 Under the hypothesis of Theorem \ref{t:localAKS}, if $T^\ast$ is the maximal time of existence of the mild solution  $\psi$, then
 \[
     \limsup_{t\to T_-^\ast} \|\phi(t)\|_{L^2(\TT^2))} = \infty.
 \]
Otherwise, $T^\ast =\infty$.
\end{corollary}

\begin{proof}
 If $T^*<\infty$ and $\limsup_{t\to T_-^\ast} \|\phi(t)\|_{L^2(\TT^2))}=\gamma<\infty$, then there exists $t_0<T^\ast$ such that for all $t_0<\Bar{t}<T^\ast$, $\|\phi(\Bar{t})\|_{L^2(\TT^2))}\leq 3\gamma/2$. Let $T$ satisfy 
 \[
    T\leq \min\Big(1, ~\frac{1}{16(3C^2\gamma+C\norm{\bv}_{L^\infty(\R_+; L^2)})^4}\Big),
\]
where $C$ is as in \eqref{eq:Tchoice}.  Choose $\Bar{t}$ such that $T>T^\ast-\Bar{t}$.
Then, by Theorem \ref{t:localAKS}, there exists a mild solution $\tilde{\phi}$ on $[\Bar{t}, \Bar{t} +T]$ with initial data $\phi(\Bar{t})$ and $\phi=\tilde{\phi}$ on $[\Bar{t},T']$ for all $\Bar{t}<T'<T^\ast$, by uniqueness of mild solutions (since, if $\phi$ is a mild solution on $[0,T']$ for all $T'<T^\ast$,  then $\sup_{\Bar{t}<t<T'} (t-\Bar{t})^{1/4}\,\|\nabla \phi(t)\|_{L^2}<\infty$). Hence, the solution can be continued past $T^\ast$, which gives a contradiction. 
\end{proof}

We close by showing that the local-in-time mild solution we constructed is actually also a weak solution on $[0,T]$. 

\begin{proposition}\label{p:mildEqualweak}
 Let $\phi$ be the mild solution on $[0,T]$ given in Theorem \ref{t:localAKS}. Then, $\phi$ is a weak solution of \eqref{e:aks} on $[0,T]$,  and satisfies the energy identity for any $0\leq t< T$:
\begin{multline} \label{eq:energyIdentity}
    \|\phi(t)\|^2_{L^2} + 2\int^{t}_{0} \|\Delta \phi(s)\|_{L^2}^2\, ds =  \|\phi_0\|^2_{L^2} \\ + 
   2 \int^{t}_{0} \|\nabla \phi(s)\|_{L^2}^2\,ds -\int^{t}_{0} \int_{\TT^2} \phi(s)\, |\nabla \phi(s)|^2\, dx\, ds.
\end{multline} 
\end{proposition}

\begin{proof}
We first show that the mild solution $\phi(t)$ can be represented as follows for $0<\epsilon \leq t \leq T$:
\begin{equation}\label{eq:Mildt0}
     \phi(t)= e^{-(t-\epsilon) \cL}\, \phi(\epsilon) -\int_{\epsilon}^t e^{-(t-\tau)\, \cL} \, \left(\frac{1}{2} |\nabla \phi(\tau)|^2+\bv(\tau) \cdot  \nabla\phi(\tau )
      \right)\,d\tau,
\end{equation}
with equality as functions in $C([\epsilon,T];L^2(\TT^2))$.
Indeed, by the semigroup property, from \eqref{eq:MildDef} it follows that:
\[
   \begin{aligned}
      \phi(t)-&\phi(\epsilon)  = (e^{-(t-\epsilon) \cL}-I) \left[e^{-\epsilon \cL}\, \phi_0 -\int^{\epsilon}_0 e^{-(\epsilon-\tau)\, \cL}  \left(
      \frac{1}{2} |\nabla \phi|^2+\bv \cdot \nabla\phi\right)(\tau) \, d\tau\right] \\
      &\qquad \qquad \qquad - \int^{t}_{\epsilon} e^{-(t-\tau)\, \cL} \, \left(\frac{1}{2}|\nabla \phi(\tau)|^2+\bv(\tau) \cdot \nabla\phi(\tau )\right)\, d\tau\\
      &=  (e^{-(t-\epsilon) \cL}-I) \,\phi(\epsilon)  - \int^{t}_{\epsilon} e^{-(t-\tau)\, \cL} \, \left(\frac{1}{2}|\nabla \phi(\tau)|^2+\bv(\tau) \cdot \nabla\phi(\tau)\right)\,d\tau,
   \end{aligned}
\]
where $I$ is the identity map.
Next, we show that $\phi\in L^2([\epsilon,T];H^2(\TT^2))$.   Since $\phi\in C([0,T];L^2(\TT^2))$, it is enough to show that $ \Delta \phi \in L^2([\epsilon,T]\times\TT^2)$.
We apply $\Delta$ to \eqref{eq:Mildt0}, and use \eqref{eq:SemigroupBound} with $s=2$, using also that $\phi \in \tX_T$ on $[\epsilon,T]$:
 \[
   \begin{aligned}
     &\|\Delta \phi(t)\|_{L^2}\leq C\,\left[ (t-\epsilon)^{-1/4} \, \|\phi(\epsilon)\|_{H^1} +
      \left(\int_{\epsilon}^t (t-\tau)^{-3/4} \,\tau^{-1/2}\, d\tau \right) \|\phi\|_{\tX_T}^2  \right .\\
      &\left. \qquad \qquad \qquad \qquad +\left(\int_{\epsilon}^t (t-\tau)^{-3/4}  \,\tau^{-1/4}\, d\tau\right) \|\bv\|_{L^\infty(\R_+;L^2)}\, \|\phi\|_{\tX_T} 
      \right]\\
      &    \leq C\, (t-\epsilon)^{-1/4} \left[\big(\epsilon^{1/4}+( t-\epsilon)^{1/4} \, \|\bv\|_{L^\infty(\R_+;L^2)}\big) \,\|\phi\|_{\tX_T} + \|\phi\|_{\tX_T}^2
        \right],\\
   \end{aligned}
 \]
 which gives the desired estimate. 
Given that $\Delta \phi$ is square integrable on $[\epsilon,T]\times \T^2$, $\frac{1}{2}|\nabla \phi(t)|^2 +\bv(t)\cdot \nabla\phi(t) \in L^1([\epsilon,T]\times \T^2)$. 
Let $\varphi \in C^\infty([\epsilon,T]\times\T^2)$ and consider the $L^2$ pairing of $\varphi$ with $\phi$. By Plancherel,
$e^{-t \cL}$ is a self-adjoint operator on $L^2$, so that
 \begin{multline} \label{eq:dualityFormRep}
      \int_{\TT^2} \varphi(t) \,\phi(t)\, dx = \int_{\TT^2} e^{-(t-\epsilon) \cL} \varphi(t)\phi(\epsilon)\, dx \\-
      \int_{\epsilon}^t \int_{\TT^2}  e^{-(t-\tau)\cL} \varphi(t)\, \left(\frac{1}{2}|\nabla \phi(\tau)|^2+\bv(\tau) \cdot \nabla\phi(\tau )\right)\, dx\, d
      \tau,
 \end{multline}
where we have used Fubini-Tonelli's Theorem to exchange the order of integration. Now, $\varphi(t)$ is in the domain of $\cL$ for all $t\in [\epsilon,T]$ by hypothesis and $\cL$ generates an analytic semigroup, hence differentiable.
Therefore, strongly in $L^2(\TT^2)$,
\[
    \lim_{h\to 0}  \frac{e^{-(t+h-\epsilon) \cL} \varphi(t)-  e^{-(t-\epsilon) \cL} \varphi(t)}{h} =
    - e^{-(t-\epsilon) \cL} \cL \varphi(t),
\]
given that $\cL \varphi(t)\in L^2(\TT^2)$ and that that $\ds \frac{d}{dt} e^{-t \cL} f = -\cL \,e^{-t \cL} f= - e^{-t \cL}\, \cL f$ for all $t\geq 0$, if $f$ is in the domain of $\cL$.
Similarly, strongly in $L^2(\TT^2)$,
\[
    \lim_{h\to 0}  \frac{e^{-(t-\epsilon) \cL} \varphi(t+h)-  e^{-(t-\epsilon) \cL} \varphi(t)}{h} =
    e^{-(t-\epsilon) \cL} \pa_t\varphi(t),
\]
by the smoothness of $\varphi$. Since $e^{-t\,\cL}$ is a strongly continuous semigroup on $L^2(\TT^2)$ it follows that
$\frac{d}{dt} e^{-(t-\epsilon) \cL} \varphi(t)\in C([\ep,T];L^2(\TT^2))$ and the Leibniz rule applies:
\begin{multline*}
   \frac{d}{dt} e^{-(t-\epsilon) \cL} \varphi(t)= (\frac{d}{dt} e^{-(t-\epsilon) \cL}) \varphi(t) \\+  e^{-(t-\epsilon) \cL} \pa_t\varphi(t) = e^{-(t-\epsilon) \cL} \pa_t \varphi(t) - \cL\,e^{-(t-\epsilon) \cL}  \varphi(t).
\end{multline*}
In particular, the pairing of $e^{-(t-\epsilon) \cL} \varphi(t)$ with any function $g(t)$ that  is absolutely continuous in $t\in [\ep,T]$ with values in $L^2(\TT^2)$, is differentiable a.e. in $t$, the derivative is integrable in time, and
\begin{multline} \label{eq:weakDerivativePairing}
  \frac{d}{dt} \int_{\TT^2} e^{-(t-\epsilon) \cL} \varphi(t) \,g(t)\, dx = \int_{\TT^2}
  \left(e^{-(t-\epsilon) \cL} \pa_t \varphi(t) g(t) - \cL\,e^{-(t-\epsilon) \cL}  \varphi(t)\, g(t) \right.\\
   \left. + e^{-(t-\epsilon) \cL} \varphi(t)\,g'(t)\right)\,dx.
\end{multline}
We are using here the property that, if a Banach space $X$ is separable and reflexive ($X=L^2(\TT^2)$ here), then absolute continuity of $X$-valued functions of time is equivalent to the existence of a weak derivative, which belongs to  $X$, for  a.e. times, and  which is integrable in the B\"ochner sense.
From \eqref{eq:dualityFormRep}, by \eqref{eq:weakDerivativePairing} with $g(t)= \phi(\ep)$ and
$g(t)= \int_{\epsilon}^t   e^{-(t-\tau)\cL} \varphi(t) \left(\frac{1}{2}|\nabla \phi(\tau)|^2+\bv(\tau) \cdot \nabla\phi(\tau )\right)d\tau$, using again Fubini-Tonelli, it follows that
\[
  \begin{aligned}
   & \frac{d}{dt} \int_{\TT^2} \varphi(t) \,\phi(t)\, dx =  \int_{\TT^2}   \left(e^{-(t-\epsilon) \cL} \pa_t \varphi(t) - \cL\,e^{-(t-\epsilon) \cL}  \varphi(t)\right) \,\phi(\epsilon)\, dx \\
    &- \int_{\TT^2}   \int_{\epsilon}^t \left(e^{-(t-\tau) \cL} \pa_t \varphi(t) - \cL\,e^{-(t-\tau) \cL} \varphi(t)\right) \, \left(\frac{1}{2}
    |\nabla \phi(\tau)|^2+\bv(\tau) \cdot\nabla \phi(\tau )\right)\,d\tau\, dx\\
    &  - \int_{\TT^2}  \varphi(t)\, \left(\frac{1}{2}|\nabla \phi(t)|^2+\bv(t) \cdot \nabla
    \phi(t)\right)\, dx.
  \end{aligned}
\]
Consequently,
\begin{align} \label{e:timedir}
\nonumber
    \frac{d}{dt}& \int_{\TT^2} \varphi(t) \,\phi(t)\, dx = \nonumber\\
    &=\int_{\TT^2} \pa_t \varphi(t) \left[ e^{-(t-\epsilon) \cL}\,\phi(\epsilon)
    -\int_{\epsilon}^t e^{-(t-\tau) \cL} \left(\frac{1}{2}|\nabla 
    \phi|^2+\bv \cdot \nabla\phi\right)(\tau) \, d\tau \right] dx \nonumber\\
    \nonumber
    &- \int_{\TT^2} \cL \,\varphi(t) \, \left[ e^{-(t-\epsilon) \cL}\,\phi(\epsilon) 
    -\int_{\epsilon}^t e^{-(t-\tau) \cL} \left(\frac{1}{2}|\nabla 
    \phi|^2+\bv \cdot\nabla \phi\right)(\tau) \, d\tau\right] \, dx \nonumber \\
    &\qquad \qquad \qquad \qquad  -\int_{\TT^2}  \varphi(t)\, \left(\frac{1}{2}|\nabla \phi(t)|^2+\bv(t) \cdot\nabla
    \phi(t)\right)\, dx \\
    & = \int_{\TT^2} (\pa_t \varphi(t) -\cL\, \varphi(t))\,\phi(t)\, dx - \int_{\TT^2}  \varphi(t)\, \left(\frac{1}{2}|\nabla \phi(t)|^2+\bv(t)   
     \cdot  \nabla\phi(t)\right)\, dx. \nonumber
\end{align}   
We integrate the above expression over $(\epsilon, t)$ for any $\epsilon < t < T$, and then integrate by parts over
$\TT^2$ twice:
\begin{align} \label{eq:differentiatedWeakForm}
 \nonumber
 \int_{\TT^2} \varphi(t)\,\phi(t) \, dx &  - \int_{\TT^2} \varphi(\epsilon)\,\phi(\epsilon) \, dx = \int^{t}_{\epsilon} \int_{\TT^2} \pa_t \varphi(\tau) \, \phi(\tau) \,dx\,d\tau \\
 \nonumber
  &  - \int^{t}_{\epsilon} \int_{\TT^2} \Delta \varphi(\tau)\,\Delta\phi(\tau)\, dx\,d\tau
    + \int^{t}_{\epsilon} \int_{\TT^2} \nabla \varphi(\tau)\cdot\nabla\phi(\tau)\, dx\,d\tau\\
    &- \int^{t}_{\epsilon} \int_{\TT^2}  \varphi(\tau)\, \left(\frac{1}{2}|\nabla 
    \phi(\tau)|^2+\bv(\tau) \cdot \nabla \phi(\tau)\right)\, dx\, d\tau,
 \end{align}
 which is justified, since $\phi\in L^2((\ep,T);H^2)\cap C([0,T];L^2)$.
From \eqref{eq:differentiatedWeakForm}, we conclude that  $\pa_t \phi \in L^2([\epsilon,T];H^{-2})$ by density.
In fact, all terms in \eqref{eq:differentiatedWeakForm} are well defined for $\varphi \in H^1((\ep,T);H^2)$. Here, by $H^1(I,X)$, where $I\subset \RR$ is an interval and $X$ is a Banach space, we mean the space of all functions $\varphi \in L^2(I;X)$ with weak derivative $\pa_t\varphi \in L^2(I;X)$.

Since $\phi \in L^2((\ep,T);H^2)\cap H^1((\ep,T);H^{-2})$, there exists a sequence
$\{\varphi^n\}\subset C^1([\ep,T];H^2)$
such that $\varphi^n\to \phi$ strongly in $C([\epsilon,T];L^2)\cap L^2((\epsilon,T);H^2)$ as $n\to \infty$ and such that $\pa_t\varphi_n \to \pa_t \phi$ weakly in $L^2((\ep,T);H^{-2})$ (this can be shown by suitably extending $\phi$ in $t$ to $\RR$ and mollifying in time).
By  the Sobolev Embedding Theorem, $\varphi^n\to \phi$ strongly in $L^2((\epsilon, T);L^4)$. By H\"older's inequality then,
\begin{multline*}
     \hspace*{-.15in}\left| \int_{\epsilon}^t \int_{\TT^2}  (\varphi^n-\phi) \left(\frac{1}{2}|\nabla \phi|^2 + \bv\cdot \nabla \phi \right)  dx\, d\tau \right|
    \leq \|\varphi^n-\phi\|_{L^\infty((\epsilon,T),L^2)} \,\|\nabla\phi\|_{L^2((\epsilon, T);L^4)}^2  \\ 
    +   \|\varphi^n-\phi\|_{L^2((\epsilon,T),L^4)} \, \|\bv\|_{L^\infty(\R_+;L^2)} \, \|\nabla \phi\|_{L^2((\epsilon, T);L^4)},
\end{multline*}
We take $\varphi^n$ as test function in \eqref{eq:differentiatedWeakForm}. We can then pass to  the limit $n\to \infty$ in every term of the resulting expression. In particular:
\begin{align}\label{eq:differentiatedWeakForm2}
&\|\phi(t)\|_{L^2}^2 -\|\phi(\epsilon)\|^2_{L^2}=\int^{t}_{\epsilon} \int_{\TT^2} \phi(\tau) \,\pa_t  \phi(\tau) \,dx\,d\tau  
    - \int^{t}_{\epsilon} \norm{\Delta\phi(\tau)}_{L^2}^2\,d\tau\\
    &+ \int^{t}_{\epsilon}  \norm{\nabla\phi(\tau)}_{L^2}^2\,d\tau
    - \int^{t}_{\epsilon} \int_{\TT^2}  \phi(\tau)\, \left(\frac{1}{2}|\nabla
    \phi(\tau)|^2+\bv(\tau) \cdot \nabla \phi(\tau)\right)\, dx\, d\tau. \nonumber
\end{align}
Since  $\phi\in L^2((\epsilon,T);H^2)$ with $\partial_t\phi \in L^2((\epsilon, T);H^{-2})$, the mapping $t\to \norm{\phi(t)}_{L^2(\T^2)}^2$ is absolutely continuous and the Fundamental Theorem of Calculus applies:
\[
\begin{aligned}
   \int_{\ep}^t \int_{\T^2} \phi(t) \partial_t \phi(t)\,dt&=\frac{1}{2}\int_\ep^t \frac{d}{dt}\norm{\phi}_{L^2(\T^2)}^2  \\
    &= \frac{1}{2} \left( \|\phi(t)\|_{L^2}^2 -\|\phi(\ep)\|^2_{L^2} \right), \qquad \epsilon < t<T.
\end{aligned}
\]
Therefore, we have from \eqref{eq:differentiatedWeakForm2}:
\begin{multline}\label{e:energytoepsilon}
  \|\phi(t)\|^2_{L^2}-\|\phi(\epsilon)\|^2_{L^2}+2\int_{\epsilon}^t
    \|\Delta \phi(\tau)\|^2_{L^2}\, d\tau =   
  2 \int_{\epsilon}^t  \|\nabla \phi(\tau)\|^2_{L^2}\,\, d\tau \\
   - \int_{\epsilon}^t\int_{\TT^2}  \phi\, |\nabla \phi|^2 dx\, d\tau\,.
\end{multline}
The term with $\bv$ vanishes identically, as $\bv$ is divergence free.

Next, we show that we can take $\epsilon \to 0$ in \eqref{e:energytoepsilon}, obtaining an energy inequality valid on $[0,T)$. First, by the  Gagliardo-Nirenberg and Young's  inequalities,
\begin{align*}
\abs[\Big]{ \int_{\epsilon}^t\int_{\TT^2}  \phi\, |\nabla \phi|^2 dx\, d\tau }&\leq C\int_{\epsilon}^t \norm{\phi(\tau)}_{L^2}^{3/2}\norm{\Delta \phi(\tau)}_{L^2}^{3/2}+ \norm{\phi}_{L^2}^3\,d\tau \\
 &\leq \int_\epsilon^t \norm{\Delta \phi(\tau)}_{L^2}^2\,d\tau +C\int_{\epsilon}^t  \norm{\phi(\tau)}_{L^2}^6 +\norm{\phi(\tau)}_{L^2}^3\,d\tau\,,
\end{align*}
where we used that  $\phi$ is a mild solution in $\tX_T$, so the left-hand side is well defined.
Utilizing this inequality in~\eqref{e:energytoepsilon},  we get
\begin{multline*}
   \int_\epsilon ^t\|\Delta \phi(\tau)\|^2_{L^2}\, d\tau \leq  \|\phi(\ep)\|^2_{L^2}-\|\phi(t)\|^2_{L^2}+
  2 \int_{\epsilon}^t  \|\nabla \phi(\tau)\|^2_{L^2}\,\, d\tau \\
   + C\int_{\epsilon}^t  \norm{\phi(\tau)}_{L^2}^6 +\norm{\phi(\tau)}_{L^2}^3\,d\tau\,.\nonumber
\end{multline*}
Since $\phi\in C([0,T], L^2)$,   $\norm{\phi(\epsilon)}_{L^2}\to \norm{\phi_0}_{L^2}$ as $\ep\to 0$, so that the right hand side in the expression above is bounded above uniformly in $\ep$. By the Monotone Convergence Theorem,
$\phi \in L^2((0,T); H^2)$ and  \eqref{eq:energyIdentity} is recovered by sending $\epsilon \to 0$ in~\eqref{e:energytoepsilon}. Similarly, since $\phi\in L^2((0,T);H^2)$, the right-hand side of   \eqref{eq:differentiatedWeakForm} with $\varphi \in L^2((0,T);H^2)$ is uniformly bounded in $\ep$, so the left-hand side is, which implies that
$(\pa_t \phi) \,\chi_{[\ep,T)}$ converges weakly to $\pa_t \phi$ in $L^2((0,T);H^{-2})$. Therefore,
by the Dominated Convergence Theorem, we can pass to the limit $\ep\to 0$ in \eqref{eq:differentiatedWeakForm}, taking the test function $\varphi\in C_c^\infty([0,T]\times \T^2)$ with support in $[0,T)$. In the limit, we recover the weak formulation~\eqref{eq:weakDef}.  Hence $\phi$ is a weak solution on $[0,T]$.
\end{proof}

\subsection{Global existence with advecting flows of small dissipation time} \label{s:globalExistence}

In this section, we tackle the issue of global existence for solutions of \eqref{e:aks}. Here, we make stronger assumptions on the advecting velocity, namely, we assume that $\bv$ is Lipschitz continuous in space uniformly in time, 
$\bv \in L^{\infty}([0,\infty),W^{1,\infty}(\TT^2))$ and continue to assume that $\bv$ is divergence free.

Let $\mathcal{S}_{s, t}$, $0\leq s\leq t$, be the solution operator of the advection-hyperdiffusion equation
\begin{align} \label{eq:linearPart}
\partial_t f + \bv \cdot \nabla f+\Delta^2 f=0.
\end{align}
That is, $\mathcal{S}_{s,t}$ maps the solution of the above equation at time $s$ to the solution at time $t\geq s$. The solution operator satisfies the semiflow property and form an evolution system (cf. \cite{Lunardi}). 

We introduce the concept of dissipation time for \eqref{eq:linearPart} (see \cite{FFITArXiv20}). Informally, it is the time it takes the system to reduce the energy, i.e., the norm of the solution in $L^2(\TT^2)$, by half. We give the precise definition below.

\begin{definition} \label{d:dissipationTime}
The number $0<\tau^\ast<\infty$, where
\begin{align}
\tau^*=\inf \set[\big]{t\ge 0 \st \norm{\mathcal{S }_{s,s+t}}_{\rL^2 \to \rL^2}\leq \frac{1}{2}\,, \text{ for all $s\geq 0$}},
\end{align}
is called the {\em dissipation time}  associated to the system $\mathcal{S}_{s,t}$, $0\leq s\leq t$. With slight abuse of notation, we will also refer to $\tau^\ast$ as the dissipation time of $\bv$.
\end{definition}

We will show that the mild solution constructed in Section \ref{s:localExistence} can be continued for any time $t> T$, provided the flow generated by the velocity field $\bv$ induces a  sufficiently small dissipation time. Indeed, the proof of Theorem \eqref{t:localAKS} shows that, given the initial data for the problem, the time of existence of the mild solution depends only on the $L^2$-norm of the initial data.

\begin{remark} \label{r:diffusionTime}
When $\bv$ is more regular, $\bv\in L^\infty([0,\infty),C^2(\TT^2))$, $\tau^\ast$ can become arbitrarily small, provided the amplitude of $\bv$ is large enough and $\bv$ is {\em weakly mixing} (see \cite[Proposition 1.4, Definition 3.1]{FFITArXiv20}).
While there are no known explicit examples of optimal mixing flows with $C^2$ regularity in two dimensions  (for an example  in the Lipschitz class, we refer to \cite{ACM19}, see also \cite{EZ19}), many examples can be given using probabilistic methods \cite{BBPSArXiv19}.
\end{remark}

In the definition of $\tau^\ast$ we have implicitly used the fact that the projection $\PP$ onto  mean-free functions commutes with $S_{s,t}$. 
To control the $L^2$-norm of $\phi(t)$, we employ energy estimates, after we project onto the subspace of mean-free elements. Indeed, the mean of a non-constant solution  grows exponentially in time.
We denote the mean of $\phi \in L^p(\TT^2)$ by $\bar \phi=\int_{\T^2}\phi(x)\,dx$ and set $\psi=\PP \phi = \phi-\bar \phi\in \rL^2(\TT^2)$, then $\bar\phi$ formally satisfies 
\begin{align}\label{e:mean}
\frac{d}{dt}\bar \phi=-\frac{1}{2}\norm{\nabla \phi}_{L^2}^2= -\frac{1}{2}\norm{\nabla \psi}_{L^2}^2\,,
\end{align}
and $\psi$ formally satisfies
\begin{align}\label{e:mks}
\partial_t \psi+\bv \cdot \nabla \psi+\frac{1}{2}|\nabla \psi|^2-\frac{1}{2}\norm{\nabla \psi}_{L^2}^2+\Delta^2 \psi+\Delta \psi=0\,.
\end{align}
Recall that the mild solution $\phi$ satisfies $\sup_{0<t<T} t^{1/4}\, \|\nabla \phi(t)\|_{L^2}<\infty$. Therefore, $\|\nabla \phi(t)\|^2_{L^2}$ is integrable  in $(0,T)$, so that  \eqref{e:mean} holds at least a.e. in $(0,T)$. It follows, using arguments similar to those in the proof of  Proposition \ref{p:mildEqualweak}, that $\psi$ is a weak solution of \eqref{e:mks} on $[0, T]$.

We begin by deriving {\em a priori} energy estimates for $\psi$ . Let $N(\psi)=\frac{1}{2}|\nabla \psi|^2-\frac{1}{2}\norm{\nabla \psi}_{L^2}^2+\Delta \psi$. Then~\eqref{e:mks} becomes
\begin{align}
\partial_t \psi + \bv \cdot \nabla \psi+\Delta^2 \psi+ N(\psi)=0\,.
\end{align}

\subsubsection{Bounds on the non-linear operator $N$} \label{s:Nestimates}
We observe that the following estimates hold for the operator $N$.
\begin{itemize}[]
\item[(1)] For any $\psi\in \rH^2(\T^2)$, there exists a constant $C=C(L_1,L_2)$ such that 
\begin{align}\label{e:Nenergy}
\abs[\Big]{\int_{\T^2}\psi N(\psi)\,dx}& \leq \frac{1}{2}\norm{\Delta \psi}_{L^2}^2+C\norm{\psi}_{L^2}^2+C\norm{\psi}_{L^2}^6\,.
\end{align}
In fact, we first have by Cauchy's inequality with $\epsilon$ and by H\"older's inequality:
\begin{align*}
\abs[\Big]{\int_{\T^2}\psi N(\psi)\,dx}& \leq \frac{1}{2}\abs[\Big]{\int\psi |\nabla \psi|^2\,dx}+\abs[\Big]{\int \psi\Delta \psi\,dx}+\frac{1}{2}\norm{\nabla \psi}_{L^2}^2\norm{\psi}_{L^1}  \\
&\leq \frac{1}{2}\norm{\psi}_{L^2}\norm{\nabla \psi}_{L^4}^2+ \norm{\psi}_{L^2}^2+\frac{1}{4}\norm{\Delta \psi}_{L^2}^2+\frac{\sqrt{L_1L_2}}{2}\norm{\nabla \psi}_{L^2}^2 \norm{\psi}_{L^2}\,.
\end{align*}
We then use the Gagliardo-Nirenberg interpolation inequality 
\begin{align*}
\norm{\nabla \psi}_{L^4}&\leq C\norm{\Delta\psi}_{L^2}^{3/4}\norm{\psi}_{L^2}^{1/4}\,,\\
\norm{\nabla \psi}_{L^2}& \leq C\norm{\Delta\psi}_{L^2}^{1/2}\norm{\psi}_{L^2}^{1/2}\,,
\end{align*}
and Young's inequality with $\epsilon$ to get the desired result.

\item[(2)] For any $\psi\in \rH^2(\T^2)$, there exists a constant $C=C(L_1,L_2)$ such that 
\begin{align}\label{e:N}
\norm{N(\psi)}_{L^2}\leq C\norm{\Delta \psi}_{L^2}^{2}+C\norm{\Delta\psi}_{L^2}\,.
\end{align}
Indeed, Poincar\' e inequality, interpolation, and other standard inequalities give: 
\begin{align*}
\norm{N(\psi)}_{L^2} &\leq \frac{1}{2}\norm{\nabla \psi}_{L^4}^2+\frac{L_1L_2}{2}\norm{\nabla \psi}_{L^2}^2+\norm{\Delta \psi}_{L^2}\\
&\leq C\norm{\Delta \psi}_{L^2}^{3/2}\norm{\psi}_{L^2}^{1/2}+\frac{L_1L_2}{2\lambda_1}\norm{\Delta \psi}_{L^2}^2+\norm{\Delta\psi}_{L^2}\\
&\leq C\norm{\Delta \psi}_{L^2}^{2}+C\norm{\Delta\psi}_{L^2}\,,
\end{align*}
\end{itemize}
where $\lambda_1$ is the first eigenvalue of the Laplace-Beltrami operator $-\Delta$ on $\T^2$.

Exploiting these estimates, we establish that the $L^2$ norm of $\psi$ will not double within a certain interval of  time.

\begin{lemma}\label{l:nodouble}
For any $B>0$, let 
\begin{align}  \label{eq:T0Bdef}
T_0(B)=\int_{B^2}^{4B^2}\frac{1}{Cy+Cy^3}\,dy ,
\end{align}
where $C=C(L_1,L_2)$ is the constant in \eqref{e:Nenergy}-\eqref{e:N}.
Let $\psi=\PP \phi$, where $\phi$ is a mild solution of \eqref{e:aks} on $[0,T]$. 
If $\norm{\psi(t_0)}_{L^2}\leq B$, then for any $t_0\leq t\leq t_0+T_0(B)$, $0\leq t_0\leq T$, it holds
\begin{align} \label{e:nodouble}
\norm{\psi(t)}_{L^2}\leq 2\norm{\psi(t_0)}_{L^2}\,.
\end{align}
\end{lemma}

\begin{proof}
 By the energy identity for $\phi$, we derive an energy identity for $\psi$:
\begin{align*}
\frac{1}{2}\frac{d}{dt}\norm{\psi}_{L^2}^2=-\norm{\Delta \psi}_{L^2}^2+\int \psi\,N(\psi) \,dx\,.
\end{align*}
Applying~\eqref{e:Nenergy}, we have
\begin{align}\label{e:genergy}
\frac{d}{dt}\norm{\psi}_{L^2}^2 &\leq -\norm{\Delta \psi}_{L^2}^2+ C\norm{\psi}_{L^2}^2+C\norm{\psi}_{L^2}^6 \leq C\norm{\psi}_{L^2}^2+C\norm{\psi}_{L^2}^6\,.
\end{align}
By Gr\"onwall's inequality,  $\norm{\psi}_{L^2}^2$ can at most grow as the solution of the following ode
\begin{align}
\frac{d}{dt} y= C y + C y^3\,, \qquad y(t_0)=\norm{\psi(t_0)}_{L^2}^2\,.
\end{align}
which implies
\begin{align*}
\norm{\psi(t)}_{L^2}\leq 2\norm{\psi(t_0)}_{L^2}, \quad t_0\leq t\leq t_0+T_0(\norm{\psi(t_0)}_{L^2}),
\end{align*}
where   $T_0(\norm{\psi(t_0)}_{L^2})$ is defined as in \eqref{eq:T0Bdef} with $B=\|\psi(t_0)\|_{L^2}$.
Finally, since $T_0(B)$  is a decreasing function in $B$, $T_0(B)\leq T_0(\norm{\psi(t_0)}_{L^2})$ as long as $\norm{\psi(t_0)}_{L^2}\leq B$. Thus~\eqref{e:nodouble} is proved.
\end{proof}

In the previous Lemma, we neglected the action of hyperdiffusion, which is encoded in the  integral of $\|\Delta \psi(t)\|_{L^2}^2$. Next, we examine the case when this term is large on time intervals of the order of $T_0(B)$, showing that then
hyperdiffusion alone can overcome the growth of the norm of $\psi$. When this term is small, instead, we will gain decay by the enhanced diffusion of the operator $S_{s,t}$, provided $\bv$ satisfies some conditions (cf. Remark \ref{r:diffusionTime}). 

\begin{lemma}\label{l:h2large}
Let  $\mu>0$.  Given $B>0$, if $\psi$ satisfies $\norm{\psi(t_0)}_{L^2}\leq B$, and if for  $0<\tau< T_0(B)$,
\begin{align}\label{e:h2large}
\frac{1}{\tau}\int_{t_0}^{t_0+\tau} \norm{\Delta \psi(t)}_{L^2}^2\,dt \ge 2\mu \norm{\psi(t_0)}_{L^2}^2+4C \norm{\psi(t_0)}_{L^2}^2+64C \norm{\psi(t_0)}_{L^2}^6\,,
\end{align}
where $C=C(L_1,L_2)$ is the constant in \eqref{e:Nenergy}-\eqref{e:N}, then
\begin{align}\label{e:decay}
\norm{\psi(t_0+\tau)}_{L^2}\leq e^{-\mu \tau}\norm{\psi(t_0)}_{L^2}\,.
\end{align}
\end{lemma}

\begin{proof}
First, Lemma~\ref{l:nodouble} gives
\begin{align*}
\norm{\psi(t_0+t)}_{L^2}\leq 2\norm{\psi(t_0)}_{L^2}\,,\quad \forall~ 0<t\leq T_0(B)\,.
\end{align*}
Hence, $\|\psi(t)\|_{L^2}$ is uniformly bounded in the interval $[t_0,t_0+T_0(B)]$. Integrating the energy estimate~\eqref{e:genergy} over the interval $[t_0,t_0+\tau]\subset [t_0,t_0+T_0(B)]$ and using this bound, we have
\begin{equation*}
 \begin{aligned}
\norm{\psi(t_0+\tau)}_{L^2}^2&\leq \norm{\psi(t_0)}_{L^2}^2-\int_{t_0}^{t_0+\tau}\norm{\Delta \psi(t)}_{L^2}^2\,dt+4C\tau \norm{\psi(t_0)}_{L^2}^2\\
&\quad+64C\tau \norm{\psi(t_0)}_{L^2}^6\leq (1-2\mu\tau)\norm{\psi(t_0)}_{L^2}^2\leq e^{-2\mu\tau}\norm{\psi(t_0)}_{L^2}^2,
\end{aligned}
\end{equation*}
where the inequalities on the last line follows by applying hypothesis \eqref{e:h2large}.
\end{proof}

\begin{lemma} \label{l:h2small}
Let $\mu>0$. Given  $B>0$, let $T_1(B)$ be defined by
\begin{align}\label{e:T1}
T_1(B)=\frac{1}{4C(2\mu +4C+64CB^4)B+4C(2\mu +4C+64CB^4)^{1/2}},
\end{align}
where $C=C(L_1,L_2)$ is the constant in \eqref{e:Nenergy}-\eqref{e:N}.
If $\norm{\psi(t_0)}_{L^2}\leq B$ and the dissipation time of $\bv$, $\tau^\ast$, satisfies
\begin{align}\label{e:choiceoftau}
\tau^\ast\leq \min \Big( T_0(B), ~T_1(B), ~\frac{1}{4\mu}\Big) \,,
\end{align}
then~\eqref{e:decay} holds, even when
\begin{align}\label{e:h2small}
\frac{1}{\tau^\ast}\int_{t_0}^{t_0+\tau^\ast} \norm{\Delta \psi(t)}_{L^2}^2\,dt \leq  2\mu \norm{\psi(t_0)}_{L^2}^2+4C \norm{\psi(t_0)}_{L^2}^2+64C \norm{\psi(t_0)}_{L^2}^6.
\end{align}
\end{lemma}

\begin{proof}
Since $\psi$ is a mild and weak solution of \eqref{e:mks} on the interval $[t_0,t_0+\tau^\ast]$,
\begin{align}
\psi(t_0+\tau^\ast)=\mathcal{S}_{t_0, \tau^\ast} \psi(t_0)+\int_0^{\tau^\ast}\mathcal{S}_{t_0+t, t_0+\tau^\ast}N(\psi(t_0+t))\,dt\,,
\end{align}
and, by \eqref{e:N} and the definition of $\tau^\ast$, 
\begin{align*}
\norm{\psi(t_0+&\tau^\ast)}_{L^2}\leq \frac{1}{2}\norm{\psi(t_0)}_{L^2}+\int_0^{\tau^\ast}\norm{N(\psi(t_0+t))}_{L^2}\,dt\\
&\qquad \quad \leq \frac{1}{2}\norm{\psi(t_0)}_{L^2}+C\int_{t_0}^{t_0+\tau^\ast}(\norm{\Delta \psi}_{L^2}^2+\norm{\Delta \psi}_{L^2})\,dt\\
&\leq  \frac{1}{2}\norm{\psi(t_0)}_{L^2}+C\int_{t_0}^{t_0+\tau^\ast}\norm{\Delta \psi}_{L^2}^2\,dt
+C\paren[\Big]{\int_{t_0}^{t_0+\tau^\ast}\norm{\Delta \psi}_{L^2}^2\,dt}^{1/2}\sqrt{\tau^\ast}\,.
\end{align*}
Next, from the definition of $T_1$ it follows that $T_1(B)\leq T_1(\norm{\psi(t_0)}_{L^2})$. Hence $\tau^\ast\leq  T_1(\norm{\psi(t_0)}_{L^2})$ and using~\eqref{e:h2small} we get
\begin{multline*}
\norm{\psi(t_0+\tau^\ast)}_{L^2}\leq \frac{1}{2}\norm{\psi(t_0)}_{L^2}+\frac{1}{4}\norm{\psi(t_0)}_{L^2} \leq \\
\leq (1-\mu \tau^\ast)\norm{\psi(t_0)}_{L^2}
\leq e^{-\mu \tau^\ast}\norm{\psi(t_0)}_{L^2}\,,
\end{multline*}
where we used $\tau^\ast\leq \frac{1}{4\mu}$ in the second to last inequality.
\end{proof}

Our main result is an exponential decay estimate for $\psi$, which implies global existence for the mild solution $\psi$.

\begin{theorem} \label{t:globalEnergy}
Let $\mu >0$. Let $\psi$ be a mild solution of \eqref{e:mks} on $[0,T]$. For $0\leq t_0 \leq T$, let  $\norm{\psi(t_0)}_{L^2}=B>0$.
If the dissipation time of $\bv$, $\tau^\ast$, satisfies~\eqref{e:choiceoftau}, then  there exists a constant $C_0$,  such that for $t>0$,
\begin{align} \label{eq:globalEnergy}
\norm{\psi(t_0+t)}_{L^2}\leq C_0e^{-\mu t}\norm{\psi(t_0)}_{L^2}\,.
\end{align}
\end{theorem}

\begin{proof}
Take $\tau=\tau^\ast$, if~\eqref{e:h2large} holds, then applying Lemma~\ref{l:h2large}, we get 
\begin{align}\label{e:gdecay}
\norm{\psi(t_0+\tau^\ast)}_{L^2}\leq e^{-\mu \tau^\ast}\norm{\psi(t_0)}_{L^2}\,.
\end{align}
Otherwise, we apply Lemma~\ref{l:h2small} and the above estimate still holds. Iterating this estimate, we  obtain
\begin{align}
\norm{\psi(t_0+n\tau^\ast)}_{L^2}\leq e^{-\mu n\tau^\ast}\norm{\psi(t_0)}_{L^2}\,.
\end{align}
For any $t>0$, there exists $n\in \N$ such that $t \in [n\tau^\ast,~ (n+1)\tau^\ast)$. Then it holds that
\begin{align}
\norm{\psi(t_0+t)}_{L^2}\leq  e^{-\mu n\tau^\ast}\norm{\psi(t_0)}_{L^2}\leq e^{-\mu(t-\tau^\ast)}\norm{\psi(t_0)}_{L^2}\leq C_0e^{-\mu t}\norm{\psi(t_0)}_{L^2}\,,
\end{align}
where $C_0$ can be taken as $e^{1/4}$.
\end{proof}

Our goal is to prove that $\norm{\phi(t)}_{L^2}$ is uniformly bounded in $t$.

\begin{theorem}
Let $\phi$ be the solution of~\eqref{e:aks} with initial data $\phi(0)=\phi_0\in L^2(\T^2)$. Then, the {\em a priori} bound
\begin{align}
\norm{\phi(t)}_{L^2}\leq C_1,
\end{align}
holds for $t>0$ with $C_1>0$ depending on the data, but not on $t$.
\end{theorem}

\begin{proof}
We first bound the mean of $\phi$ over $\TT^2$, $\Bar\phi$. From~\eqref{e:mean}, we have:
\begin{align}\label{e:tmp1}
\abs{\bar\phi(t) -\bar\phi_0}&\leq \frac{1}{2}\int_{0}^{t}\norm{\nabla \psi(s)}_{L^2}^2\,ds
\leq  \frac{1}{2\lambda_1}\int_{0}^{t}\norm{\Delta \psi(s)}_{L^2}^2\,ds,
\end{align}
where $\lambda_1$ is the first eigenvalue of $-\Delta$ on $\TT^2$.
Estimate ~\eqref{e:gdecay} gives that $\norm{\psi(t)}_{L^2}$ decays exponentially from $\norm{\psi_0}_{L^2}$, where $\psi_0=\phi_0-\Bar\phi_0$. Therefore, 
the energy estimates~\eqref{e:genergy}  and~\eqref{eq:globalEnergy} imply that
\begin{align*}
\int_{0}^{t}\norm{\Delta \psi(s)}_{L^2}^2\,ds&\leq \norm{\psi_0}_{L^2}^2+C\int_{0}^{t}\norm{\psi(s)}_{L^2}^2\,ds+C\int_{0}^{t}\norm{\psi(s)}_{L^2}^6\,ds\\
&\leq C\norm{\psi_0}_{L^2}^2+C\norm{\psi_0}_{L^2}^6\,,
\end{align*}
 From~\eqref{e:tmp1} it follows that
\begin{align*}
  \abs{\Bar{\phi}(t)} \leq \abs{\bar\phi(t) -\bar\phi_0} +\abs{\Bar{\phi_0}} \leq  \frac{C(\norm{\psi_0}_{L^2}^2+\norm{\psi_0}_{L^2}^6)}{2\lambda_1} + \abs{\Bar{\phi_0}}.
\end{align*}
We recall that $\phi(t)=\psi(t)+\Bar{\phi}(t)$.
Then by the triangle inequality and  \eqref{eq:globalEnergy} again, we get
\begin{align*}
\norm{\phi(t)}_{L^2}\leq\norm{\psi(t)}_{L^2}+\norm{\bar\phi(t)}_{L^2}=\norm{\psi(t)}_{L^2}+L_1L_2\,\sup_{t>0}\abs{\bar\phi(t)}\\
\leq C_0e^{-\mu t}\norm{\psi_0}_{L^2}+\frac{CL_1L_2(\norm{\psi_0}_{L^2}^2+\norm{\psi_0}_{L^2}^6)}{2\lambda_1}+L_1L_2\abs{\bar\phi_0}\,,
\end{align*}
which completes the proof.
\end{proof}

Applying Corollary \ref{c:contPrinciple}, we obtain global existence of mild solutions from Theorem \ref{t:globalEnergy}.

\begin{corollary}\label{c:globalAKS}
 Under the hypothesis of Theorem \ref{t:globalEnergy}, the mild solution $\phi$ of  the advective KSE \eqref{e:aks} exists on $[0,\infty)$.
\end{corollary}

\section{Global existence without advection} \label{s:global}

In this section, we consider the standard scalar form of the KSE without advection. We prove global existence of mild solutions for small enough initial data in $L^2$ in the absence of growing modes for the linearized operator $\cL$. Global existence in $L^2$  for the differentiated form of KSE without growing modes, which corresponds to small data $\phi_0\in H^1$, was established by one of the authors and David Ambrose in \cite{AM19}. The proof in the scalar case is mode delicate, since the non-linearity is more singular. Furthermore, the spatial average of the solution is not preserved by the time evolution as for the differentiated form.

We  assume that the size of the periodic box $\TT^2$ is small:
\begin{equation} \label{eq:periodsCond}
          L_1<2\pi, \qquad L_2<2\pi,
\end{equation}
so that
\begin{equation} \label{eq:sigmaDef}
      \sigma(\tbk) = |\tbk|^4-|\tbk|^2>0, \text{ for all $\tbk \neq (0, 0)$,}
\end{equation}
since $|\tbk|>1$ by condition \eqref{eq:periodsCond} .

We let again $\psi=\PP \phi$. Using \eqref{e:mean},  it is enough to prove that $\psi$ exists globally in time and
$\|\nabla \phi(t)\|^2_{L^2} = \|\nabla\psi(t)\|^2_{L^2}$ in integrable on any fixed time interval.
To this end, we apply $\PP$  to the KSE:
\begin{equation}  \label{eq:kseProjected}
   \pa_t \psi =  -\Delta^2 \psi -\Delta \psi -\frac{1}{2} \PP(|\nabla \psi|^2).   
\end{equation}
Formally integrating in time, we obtain the mild form of the projected KSE:
\begin{equation} \label{eq:KSEmild}
   \psi(t) = e^{-t\cL} \psi_0 - \int_0^t e^{-(t-\tau)\cL} \frac{1}{2} \PP(|\nabla \psi|^2)(\tau)\, d\tau=: \cT(\psi),
\end{equation}
where $e^{-t\cL}$ denotes again the semigroup generated by the operator $\cL$,  $\psi_0=\psi(0)$, and the integral is intended in the B\"ochner sense.

We will construct a global mild solution to the KSE as a fixed point of the non-linear map $\cT$ in a suitable adapted Banach space. To this end, we introduce the global analog of the space $X_T$:
\[
       X_\infty :=\{ f:[0,\infty)\times \TT^2 \, \mid \, \sup_{t>0} t^{1/4} \|\nabla f\|_{L^2}<\infty\},
\]
and let \  $\tX_\infty=C([0,\infty);\rL^2)\cap X_\infty$ with the induced norm:
\[
     \|f\|_{\tX_{\infty}}.:= \text{Max}(\sup_{t\geq 0} \|f\|_{L^2}, \sup_{t>0} t^{1/4} \|\nabla f\|_{L^2}).
\]
The main result of this section is the following.

\begin{theorem}\label{t:KSEmildglobal}
 Let $L_1, L_2<2/\pi$ and $\psi_0\in \rL^2(\TT^2)$. There exists $\delta>0$ such that, if $\|\psi_0\|_{\rL^2}<\delta$, there exists a mild solution $\psi$ of  \eqref{eq:kseProjected} in $\tX_\infty$ such that $\psi(0)=\psi_0$.
\end{theorem}

Similarly to what done in Section \ref{s:local}, one can show that the mild solution of the projected equation is unique in $\tX_\infty$ and that it is a weak solution on $[0,\infty)$.


\begin{corollary} \label{c:KSEglobal}
 Under the condition of Theorem \ref{t:KSEmildglobal},  if $\phi_0\in L^2(\TT^2)$ and $\PP\phi_0$ is such that $\|\PP\phi_0\|_{L^2}<\delta$, for any $T>0$ there exists a mild solution of KSE with initial data $\phi_0$ on $[0,T]$.
\end{corollary}

 
We will prove both results at the end of this section. For notational ease, we use the symbol $\lesssim$ to mean $\leq c$, where the positive constant $c$ may depend on $L_1$ and $L_2$ or a regularity index $s$, but not on $t$.

\subsection{Semigroup estimates}

A main ingredient in the proof of Theorem \ref{t:KSEmildglobal} is operator estimates for $e^{-t\cL}$, valid for $t\in (0,\infty)$, which improves on the bounds obtained in \cite{AM19} and recalled in Lemma \ref{l:l2l1} and \ref{l:gradestimate}.

In what follows, we let $\tbk_0$ denote the non-zero elements on the lattice $\tilde \ZZ^2$ of minimal distance $\kappa_0$ to the origin.  We note that $\tbk_0$ depends only on $L_1$ and $L_2$ and $\kappa_0>1$.

\begin{lemma} \label{l:l2l1global}
For any $T_1>0$, there exists constants $\gamma_1,\beta>0$ depending on $L_1, L_2<2\pi$ and $T_1$ such that for all $f\in L^1(\TT^2)$ with mean zero and for all $t>0$,
  \begin{equation} \label{eq:l2l1global}
      \|e^{-t\cL}  f\|_{L^2} \leq \gamma_1 \, h_1(t)\,\|f\|_{L^1}.
  \end{equation}
where 
\[
     h_1(t)=\begin{cases}   t^{-1/4}, & 0<t\leq T_1, \\
         t^{-1/2}\, e^{-\beta t}, & t>T_1.
         \end{cases}
\]     
\end{lemma}

\begin{proof}
We proceed as in Lemma \ref{l:l2l1}. By Plancherel's theorem:
\[
     \|e^{-t\sigma(\tbk)} \what{f}(\tbk)\|_{\ell^2} \lesssim \|f\|_{L^1}\, \| e^{-t\sigma(\tbk)}\|_{\ell^2}.
\]
By McLaurin's formula,
\[
  \begin{aligned}
      \| e^{-t \sigma({\tbk}) }&\|_{\ell^2}^2\lesssim \int_{\R^2-B(\mathbf{0}, \kappa_0)} e^{-2t \sigma({\tbk})}\, d\tbk 
      \lesssim \int_{\kappa_0}^\infty \kappa \, e^{-2 t \kappa^2 (\kappa^2-1)}\, d\kappa\\
      &\lesssim \int_{\kappa_0}^\infty \kappa \, e^{-2 t \kappa^2 (\kappa_0^2-1)}\, d\kappa
      \lesssim \frac{1}{t} \, e^{-2\beta t}, 
      \end{aligned}
\]
where $\beta =\kappa_0^2(\kappa_0^2-1)$.

Next, we observe that for $0<t\leq T_1$:
\[
  \begin{aligned}
      \| &e^{-t\sigma({\tbk})}\|_{\ell^2}^2\lesssim \int_{\R^2-B(\mathbf{0}, \kappa_0)} e^{-2t \sigma({\tbk})}\, d\tbk 
      \lesssim \int_{\kappa_0}^\infty \kappa \, e^{-2 t \kappa^2 (\kappa^2-1)}\, d\kappa \\
      &\lesssim \frac{1}{t^{1/2}} \int_{t^{1/4}\kappa_0}^{\infty} \Bar{\kappa}\, e^{-2  \Bar{\kappa}^2 (\Bar{\kappa}^2-t^{1/2})}\, 
       d\Bar{\kappa}  
 \lesssim \frac{1}{t^{1/2}} \int_{t^{1/4}\kappa_0}^{\infty} \Bar{\kappa}\, e^{-2  \Bar{\kappa}^2 (\Bar{\kappa}^2-T_1^{1/2})}\, 
       d\Bar{\kappa}\\
 & \lesssim \frac{1}{t^{1/2}} \int_{0}^{\infty} \Bar{\kappa}\, e^{-2  \Bar{\kappa}^2 (\Bar{\kappa}^2-T_1^{1/2})}\, d\Bar{\kappa}
  \lesssim t^{-1/2}.
  \end{aligned}
\]
\end{proof}

\begin{lemma} \label{l:Hsl2global}
 There exists  constants $\gamma_2, T_2>0$ depending on $s, L_1, L_2<2\pi$ such that for all $f\in \rL^2(\TT^2)$   and for all $t>0$,
  \begin{equation} \label{eq:Hsl2global}
      \|e^{-t\cL}  f\|_{\rH^s} \leq \gamma_2 h_2(t)\,\|f\|_{L^2}.
  \end{equation}
where 
\[
     h_2(t)=\begin{cases}   t^{-s/4}, & 0<t\leq T_2, \\
              e^{-\beta t}, & t>T_2.
        \end{cases}
\]     
\end{lemma}

\begin{proof}
We proceed as in Lemma \ref{l:gradestimate}. By Plancherel's theorem:
\begin{multline*}
  \|e^{-t\cL}  f\|_{\rH^s}^2  \lesssim \left(\sup_{\tbk\in \tilde{\ZZ}^2, \tbk\ne 0} |\tbk|^{2s}\,
    e^{-2t\sigma(\tbk)}\right) \, \|f\|_{L^2(\TT^2)}^2 \\
    \lesssim\left(\max_{\kappa\in \RR^+} \kappa^{2s}\,
    e^{-2t(\kappa^4-\kappa^2)}\right) \, \|f\|_{L^2(\TT^2)}^2.  \qquad 
\end{multline*}
Now, we simply observe that the maximum of the function $g(\kappa):=  \kappa^{2s}\,
e^{-2t(\kappa^4-\kappa^2)}$ occurs at $\bar{\kappa}^2=(\sqrt{1+4s/t}+1)/4$. Let $t=T_2$ be such that 
$\bar{\kappa}\leq \kappa_0$ (which occurs for $T_2$ large enough). Then for $0<t\leq T_2$,
$g(\kappa)\lesssim t^{-s/2}$ for all $\kappa\in \RR^+$, while for $t>T_2$, the maximum of $g$ occurs at $\kappa_0$ for $\kappa\geq \kappa_0$, so $g(\kappa)\lesssim e^{-2\beta t}$, where $\beta = \kappa_0^4-\kappa_0^2$ as in Lemma \ref{l:l2l1global}. Therefore:
\[
   \|e^{-t\cL}  f\|_{\rH^s}^2 \lesssim \begin{cases} \, t^{-s/2}  \|f\|_{L^2(\TT^2)}^2, & 0<t\leq T_2, \\
      e^{-2\beta t} \|f\|_{L^2(\TT^2)}^2, & t> T_2.
   \end{cases}
 \]
\end{proof}

\subsubsection{Proofs}

We proceed with the proof of Theorem  \ref{t:KSEmildglobal}. We find it convenient to use the following version of the Banach Contraction Mapping Theorem (see e.g. \cite{Can04}).

\begin{proposition} \label{p:Banach}
  Let $X$ be a Banach space equipped with norm $\|\cdot\|$, and let $B:X \times X\to X$ denote a bounded bilinear operator such that for some $\eta>0$,
  \begin{equation}
         \|B(x_1,x_2)\| \leq \eta \,\|x_1\|\,\|x_2\|,  \qquad \forall x_1, x_2\in X.
  \end{equation}
  Then, for all $y\in X$ with $\|y\|<1/(4\eta)$, the equation 
  \[
                    x= y +B(x,x)
  \]
  has a solution $\bar{x}\in X$ satisfying $\|\bar{x}\|\leq 2\, \|y\|$. Such solution is unique among those for which $\|x\| \leq 1/(2\eta)$.
\end{proposition}

We apply this proposition with $X=\tX_\infty$, $y=e^{-t\cL} \psi_0$, and $B(\psi_1,\psi_2) = -\frac{1}{2}\int_0^t e^{-(t-\tau) \cL}
\PP(\nabla \psi_1(\tau)\cdot \nabla \psi_2(\tau))\,d\tau$, so that $\mathcal{T}(\psi)= y +B(\psi,\psi)$. From Lemma \ref{l:l2l1global} and \ref{l:Hsl2global}, it follows immediately that $y\in \tX_\infty$.

We next show that $B$ is well defined and bounded.

\begin{lemma} \label{l:Bmapping1}
  $B:\tX_\infty\times \tX_\infty \to C([0,\infty);\rL^2)$ and there exists $\eta_1>0$ such that:
  \[
        \|B(\psi_1,\psi_2)(t)\|_{\rL^2} \leq \eta_1 \,\|\psi_1\|_{\tX_\infty}\,\|\psi_2\|_{\tX_\infty},
  \]
  for all $t\geq 0$ and for all $\psi_i\in \tX_\infty$, $i=1,2$.
\end{lemma}

\begin{proof} We first observe that, in view of the definition of the B\"ochner integral, it is enough to show that  $\|e^{-(t-\tau) \cL} \PP(\nabla \psi_1(\tau)\cdot \nabla \psi_2(\tau))\|_{\rL^2}\in L^1((0,t))$ uniformly in $t$.  We set $T_0:=\max(T_1,T_2)$ as in lemmas~\ref{l:l2l1global} and~\ref{l:Hsl2global}, and distinguish two cases.
\begin{enumerate}[label=(\alph*), ref=(\alph*)]
\item \label{i:a} $0< t\leq T_0$: by Lemma~\ref{l:l2l1global}, 
\[
      h_1(t) = t^{-1/4}, \qquad 0<t<T_0\,.
\]  
Hence:
\[
  \begin{aligned}
     \|B(\psi_1,\psi_2)(t)\|_{\rL^2} &\leq \frac{\gamma_1}{2} \, \int_0^t \frac{1}{(t-\tau)^{1/4}} \| (\nabla \psi_1(\tau)\cdot \nabla \psi_2(\tau))
     \|_{L^1}\, d\tau \\
     &\leq\frac{\gamma_1}{2} \left(\int_0^t \frac{1}{(t-\tau)^{1/4}} \frac{1}{\tau^{1/2}}\, d\tau\right) \| \psi_1\|_{\tX_\infty} \, \|\psi_2
     \|_{\tX_\infty} \\
     &\leq \tilde\gamma \, t^{1/4} \, \| \psi_1\|_{\tX_\infty} \, \|\psi_2\|_{\tX_\infty},
  \end{aligned}
\]
where the last inequality is obtained by making the change of  variable $\theta =\tau/t$ in the integral.

\item\label{i:b} $t>T_0$: we split the time integration into two parts:
\begin{multline*}
     \qquad  B(\psi_1,\psi_2)(t) = \frac{1}{2}\int_0^{t-T_0} e^{-(t-\tau) \cL}
\PP(\nabla \psi_1(\tau)\cdot \nabla \psi_2(\tau))\,d\tau \\ 
    + \frac{1}{2} \int_{t-T_0}^t e^{-(t-\tau) \cL}
\PP(\nabla \psi_1(\tau)\cdot \nabla \psi_2(\tau))\,d\tau =: I_1 +I_2. \qquad 
\end{multline*}
Since in $I_1$, $t-\tau>T_0$, from Lemma \ref{l:l2l1global}
\[
  \begin{aligned}
   \|I_1\|_{\rL^2} & \leq \frac{\gamma_1}{2}  \int_0^{t-T_0}   (t-\tau)^{-1/2}\, e^{-\beta (t-\tau)}\, \|\nabla \psi_1(\tau)\cdot     
     \nabla  \psi_2(\tau)\|_{L^1}\,d\tau\\
     &\lesssim \left(\int_0^{t-T_0}   (t-\tau)^{-1/2}\, \tau^{-1/2}\, e^{-\beta (t-\tau)}\,d\tau\right) \,\| \psi_1\|_{\tX_\infty} \, 
     \|\psi_2 \|_{\tX_\infty} \\
     & \lesssim \left(\int_0^{t}   (t-\tau)^{-1/2}\, \tau^{-1/2}\,d\tau\right) \,\| \psi_1\|_{\tX_\infty} \, 
     \|\psi_2 \|_{\tX_\infty} \\
     & \leq \Bar{\gamma} \,\| \psi_1\|_{\tX_\infty} \, 
     \|\psi_2 \|_{\tX_\infty}
 \end{aligned}
\]
again by making a change of variables in the last integral.
We next bound $I_2$. Since in $I_2$, $t-\tau<T_0$, as in Case \ref{i:a},
\[
  \begin{aligned}
    \|I_2\|_{\rL^2} &\leq 
      \frac{\gamma_1}{2} \int_{t-T_0}^t \frac{1}{(t-\tau)^{1/4}} \| \nabla \psi_1(\tau)\cdot \nabla \psi_2(\tau)
     \|_{L^1}\, d\tau \\
     &\lesssim \left(\int_{t-T_0}^t \frac{1}{(t-\tau)^{1/4}\,\tau^{1/2}}\, d\tau\right) \,\| \psi_1\|_{\tX_\infty} \, \|\psi_2
     \|_{\tX_\infty} \\
    &\lesssim \left(\int_{0}^{T_0} \frac{1}{\bar\tau^{1/4}\,(t-\bar\tau)^{1/2}}\, d\bar\tau\right)\,\| \psi_1\|_{\tX_\infty} \, \|    
    \psi_2\|_{\tX_\infty} \\
     &\leq \mathring{\gamma} \,T_0^{1/4} \,\| \psi_1\|_{\tX_\infty} \, \|
     \psi_2\|_{\tX_\infty}
  \end{aligned}
\]
where $\bar\tau=t-\tau$ and the last inequality follows by making the change of variables in the integral $\theta=\bar\tau/T_0$.
\end{enumerate}
The desired estimate now follows by taking the supremum over $t$ in both cases and setting  $\eta_1=\max(\tilde\gamma\, T_0^{1/4},\bar\gamma, \mathring{\gamma} \,T_0^{1/4})$.
\end{proof}

\begin{lemma} \label{l:Bmapping2}
  There exists $\eta_2>0$ such that:
  \[
        t^{1/4}\,\|B(\psi_1,\psi_2)(t)\|_{\rH^1} \leq \eta_2  \,\|\psi_1\|_{\tX_\infty}\,\|\psi_2\|_{\tX_\infty},
  \]
  for all $t>0$ and for all $\psi_i\in \tX_\infty$, $i=1,2$.
\end{lemma}

\begin{proof} We again distinguish two cases and use the notation in the proof of Lemma \ref{l:Bmapping1}.

\begin{enumerate}[start=3, label=(\alph*), ref=(\alph*)]
\item \label{i:c} $0< t\leq T_0$: 
we note first that
\begin{multline*}
   \qquad \|B(\psi_1,\psi_2)(t)\|_{\rH^1} \leq\frac{1}{2}\int_0^t \|e^{-\frac{(t-\tau)}{2} \cL}\|_{L^2->\rH^1} \\
   \| e^{\frac{-(t-\tau)}{2} 
     \cL}\|_{L^1->L^2} \| \nabla \psi_1(\tau)\cdot \nabla 
    \psi_2(\tau)\|_{L^1}\,d\tau. \qquad
\end{multline*}
As in Case \ref{i:a},
\[
      h_1(t)=h_2(t) = t^{-1/4}, \qquad 0<t<T_0,
\]  
where $h_1, h_2$ are the functions in Lemmas~\ref{l:l2l1global} and~\ref{l:Hsl2global}.  We then have the bound:
\[
  \begin{aligned}
     \|B(\psi_1,\psi_2)(t)\|_{\rH^1} &\leq \frac{\gamma_1\gamma_2}{\sqrt 2} \, \int_0^t \frac{1}{(t-\tau)^{1/2}} \| \nabla \psi_1(\tau)\cdot \nabla \psi_2(\tau)
     \|_{L^1}\, d\tau \\
     &\leq  \frac{\gamma_1\gamma_2}{\sqrt 2} \left(\int_0^t \frac{1}{(t-\tau)^{1/2}} \frac{1}{\tau^{1/2}}\, d\tau\right) \| \psi_1\|_{\tX_\infty} \, \|\psi_2
     \|_{\tX_\infty} \\
     &\leq \what\gamma \, \| \psi_1\|_{\tX_\infty} \, \|\psi_2\|_{\tX_\infty},
  \end{aligned}
\]
again by making the change of variables $\theta=\tau/t$ in the last integral.
Therefore, for $0<t\leq T_0$:
\[
    t^{1/4}\, \|B(\psi_1,\psi_2)(t)\|_{\rH^1}  \leq \what\gamma\, T_0^{1/4} \, \| \psi_1\|_{\tX_\infty} \, \|\psi_2\|_{\tX_\infty}.
\] 

\item \label{i:d} $t>T_0$: again we split the integral into two parts:
\begin{multline*}
     \qquad  \nabla B(\psi_1,\psi_2)(t) = \frac{1}{2}\int_0^{t-T_0} \nabla e^{-(t-\tau) \cL}
\PP(\nabla \psi_1(\tau)\cdot \nabla \psi_2(\tau))\,d\tau \\ 
    + \frac{1}{2}\int_{t-T_0}^t \nabla e^{-(t-\tau) \cL}
\PP(\nabla \psi_1(\tau)\cdot \nabla \psi_2(\tau))\,d\tau =: \tilde{I}_1 +\tilde{I}_2. \qquad 
\end{multline*}
Since in $\tilde{I}_1$, $t-\tau>T_0$, from Lemma \ref{l:l2l1global}-\ref{l:Hsl2global} with $s=1$,
\[
  \begin{aligned}
     t^{1/4} \|\tilde{I}_1\|_{\rL^2} &\leq  t^{1/4} \cdot \\     
      \cdot \int_0^{t-T_0} &
   \|e^{-\frac{(t-\tau)}{2} \cL}\|_{L^2\to\rH^1} \| e^{-\frac{(t-\tau)}{2} 
     \cL}\|_{L^1\to L^2}  \| \nabla \psi_1(\tau)\cdot \nabla 
    \psi_2(\tau)\|_{L^1}\,d\tau\\ 
    &\lesssim  \left( t^{1/4} \int_0^{t-T_0} \frac{e^{-\beta (t-\tau)}}{(t-\tau)^{1/2} \, \tau^{1/2}}\, d\tau   
    \right)  \| \psi_1\|_{\tX_\infty} \, \|\psi_2 \|_{\tX_\infty} \\
     &\lesssim\, \left( t^{1/4} \int_{T_0}^t \frac{e^{-\beta \bar\tau}}{(t-\bar\tau)^{1/2} \, \bar\tau^{1/2}}\, 
     d\bar\tau \right)  \| \psi_1\|_{\tX_\infty} \, \|\psi_2 \|_{\tX_\infty} \\
     &\lesssim\, \left( t^{1/4} \int_0^t \frac{e^{-\beta \bar\tau}}{(t-\bar\tau)^{1/2} \, \bar\tau^{1/2}}\, 
     d\bar\tau \right)  \| \psi_1\|_{\tX_\infty} \, \|\psi_2 \|_{\tX_\infty} \\
     &\lesssim \left( t^{1/4} \int_0^1 \frac{e^{-\beta t \theta}}{(1-\theta)^{1/2} \, \theta^{1/2}}\, 
     d\theta \right)  \| \psi_1\|_{\tX_\infty} \, \|\psi_2 \|_{\tX_\infty} \\
     &\lesssim \left( t^{1/4}  e^{-\beta t/2} \,\mathcal{I}_0(\beta t/2) \right)  \| \psi_1\|_{\tX_\infty} \, \|\psi_2 \|_{\tX_\infty}, \\
  \end{aligned}
\]
where $\bar\tau=t-\tau$, $\theta=\Bar\tau/t$, and $\mathcal{I}_0$ is the modified Bessel function of the first kind. Since $\mathcal{I}_0(t)=O(e^t/\sqrt{t})$ for large $t$, we conclude that $\sup_{t>T_0}  t^{1/4}  e^{-\beta t/2} \,\mathcal{I}_0(\beta t/2)\leq c(L_1,L_2)$, so that 
\[
     \sup_{t>T_0} t^{1/4}\,\|\tilde{I}_1\|_{L^2} \leq \alpha_1(L_1,L_2)\,  \| \psi_1\|_{\tX_\infty} \, \|\psi_2 \|_{\tX_\infty}.
\] 
We next turn to the second integral. Since in $\tilde{I}_2$, $t-\tau<T_0$, from Lemma \ref{l:l2l1global}-\ref{l:Hsl2global} with $s=1$,
\[
  \begin{aligned}
    t^{1/4}\, &\|I_2\|_{\rL^2} \leq 
      \frac{\gamma_1 \gamma_2}{\sqrt 2} \, t^{1/4}\, \int_{t-T_0}^t \frac{1}{(t-\tau)^{1/2}} \| \nabla \psi_1(\tau)\cdot \nabla \psi_2(\tau)
     \|_{L^1}\, d\tau \\
     &\lesssim   t^{1/4} \left( \int_{0}^{T_0} \frac{1}{(t-\bar{\tau})^{1/2}\, \bar{\tau}^{1/2}}\, d\bar{\tau} \right) 
     \| \psi_1\|_{\tX_\infty}  \, \|\psi_2 \|_{\tX_\infty} \\
      &\lesssim    t^{1/4} \left ( \int_{0}^{T_0/t} \frac{1}{(1-\theta)^{1/2}\theta^{1/2}}\, d\theta\right) \| \psi_1\|_{\tX_\infty} \, 
      \|\psi_2 \|_{\tX_\infty} \\
      & \lesssim   t^{1/4} \arcsin\left(\sqrt{T_0/t}\right) \, \| \psi_1\|_{\tX_\infty} \, \|\psi_2\|_{\tX_\infty},
  \end{aligned}
\]
where we used the same notation as for $\tilde{I}_1$ and that $0<T_0/t<1$. Since $\arcsin(x)=O(x)$ as $x\to 0$, we conclude that \mbox{$\sup_{t>T_0} t^{1/4}$} $\arcsin\left(\sqrt{T_0/t}\right)< c(L_1,L_2)$, so that
\[
     \sup_{t>T_0} t^{1/4}\,\|\tilde{I}_2\|_{L^2} \leq \alpha_2(L_1,L_2)\,  \| \psi_1\|_{\tX_\infty} \, \|\psi_2 \|_{\tX_\infty}.
\] 
\end{enumerate}

We finally obtain the desired estimate by setting $\eta_2=\max(\what\gamma\, T_0^{1/4}, \alpha_1, \alpha_2)$.
\end{proof}

\begin{proof}[Proof of Theorem \ref{t:KSEmildglobal}]
From Lemma \ref{l:Bmapping1} and \ref{l:Bmapping2}, it follows that $B:\tX_\infty\times \tX_\infty \to \tX_\infty$ and that
\begin{equation} \label{eq:Bmapping}
     \|B(\psi_1,\psi_2)\|_{\tX_\infty} \leq \eta \,\|\psi_1\|_{\tX_\infty}\,\|\psi_2\|_{\tX_\infty},
\end{equation}
for all $\psi_i\in \tX_\infty$, $i=1,2$, where $\eta=\max(\eta_1,\eta_2)$.
Next, Lemma~\ref{l:Hsl2global} gives that $\|e^{-t \cL} \psi_0\|_{\tX_\infty}<1/(4\eta)$
provided   $\|\psi_0\|_{\rL^2} <\delta$ with $\delta$ sufficiently small. We conclude the proof by applying Proposition \ref{p:Banach}, observing that $\mathcal{T}(\psi)$ is continuous in $t$ in $\rL^2$ so that 
$\psi(0)=\mathcal{T}(\psi)(0)=\psi_0$.
\end{proof}

\begin{proof}[Proof of Corollary \ref{c:KSEglobal}]
Let $\psi_0=\PP \phi_0$. Let $\psi$ be the mild solution obtained in Theorem \ref{t:KSEmildglobal}. Fix $T>0$ arbitrary. Then, $\psi\in C([0,T];\rL^2)\cap \tX_T$ is a mild solution of the projected equation \eqref{eq:kseProjected} on $[0,T]$
Let $\Bar\phi(t)$ be the unique solution of
\[
      \frac{\partial \Bar{\phi}  }{\pa t} = -\|\nabla \psi\|_{L^2}^2, \qquad 0<t<T,
\]
with $\Bar\phi(0)=\Bar\phi_0$ and let $\phi=\psi+\Bar{\phi}$. Then,  $\phi\in C([0,T];L^2)\cap \tX_T$ and $\Bar{\phi}$ solves \eqref{e:mean}, consequently $\phi$ is a mild solution of KSE on $[0,T]$.
\end{proof}

\bibliographystyle{plain}
\bibliography{KSmixing}

\end{document}